\documentclass[11pt,twoside,a4paper]{article}
\usepackage[english]{babel}
\usepackage{authblk}

\usepackage{a4wide}
\usepackage{graphicx}
\usepackage{amsmath,amssymb,amsthm, amsgen}
\usepackage{listings}
\usepackage[normalem]{ulem}
\usepackage{color}
\usepackage[usenames,dvipsnames]{xcolor}
\usepackage{rotating}
\usepackage{hhline}
\usepackage{multirow}
\usepackage{enumerate}
\usepackage[bookmarks]{}
\usepackage{amsfonts}   
\usepackage{dsfont}
\usepackage{tikz}
\usepackage{esint}
\usepackage{booktabs}
\usepackage{stmaryrd}
\usepackage{url}
\usepackage{framed}

\usepackage{bm, soul}
\usepackage{sidecap}
\usepackage{wasysym}
\usepackage{color}
\usepackage{amstext}

\long\def\PATRICK#1{{\color{blue}#1}}
\long\def\pvm#1{{\color{blue}#1}}
\long\def\plc#1{{\color{red}#1}}

\newtheorem{thm}{Theorem}[section]

\newtheorem{prop}[thm]{Proposition}

\newtheorem{lem}[thm]{Lemma}

\newtheorem{rem}[thm]{Remark}

\newcommand{\argmin}{\operatorname*{arg\,min}}

\newcommand{\Dd}{\mathbb D}
\newcommand{\dist}{\operatorname{dist}}

\newcommand{\e}{\varepsilon}

\newcommand{\Lip}{\operatorname{Lip}}

\newcommand{\N}{\mathbb N}

\newcommand{\ov}[1]{ \overline{#1} }
\newcommand{\R}{\mathbb R}
\newcommand{\Ss}{\mathbb S}
\newcommand{\sign}{\operatorname{sgn}}
\newcommand{\supp}{\operatorname{supp}}

\newcommand{\weakto}{\rightharpoonup}
\newcommand{\xto}[1]{\xrightarrow{ #1 }}

\newcommand{\Z}{\mathbb Z}

\newcommand{\bB}{\mathbf B}
\newcommand{\bH}{\mathbf H}
\newcommand{\bU}{\mathbf U}

\newcommand{\cB}{\mathcal B}
\newcommand{\cE}{\mathcal E}

\newcommand{\cW}{\mathcal W}

\begin{document}
 
\title{Discrete-to-continuum limits of planar disclinations}

\author{Pierluigi Cesana}

\affil[1]{Institute of Mathematics for Industry, Kyushu University, 
Japan   and Department of Mathematics and Statistics, La Trobe University, 
Australia
}

\author{Patrick van Meurs}

\affil[2]{Faculty of Mathematics and Physics, Kanazawa University, 
Japan 
}


 \maketitle

\begin{abstract}
In materials science, wedge disclinations are defects caused by angular mismatches in the crystallographic lattice. To describe such disclinations, we introduce an atomistic model in planar domains. This model is given by a nearest-neighbor-type energy for the atomic bonds with an additional term to penalize change in volume. We enforce the appearance of disclinations by means of a special boundary condition.

Our main result is the discrete-to-continuum limit 
 of this energy as the lattice size tends to zero. Our proof method is relaxation of the energy. The main mathematical novelty of our proof is a density theorem for the special boundary condition. In addition to our limit theorem, we 
  construct examples of   planar disclinations as  solutions to numerical
 minimization of the model  and show that
 classical results for   wedge disclinations
 are recovered by our analysis. 
\end{abstract}

\section{Introduction}

\textcolor{black}{This paper 
is devoted to the mathematical analysis of a discrete model that describes 
frustrations in atomistic lattices induced by
rotational mismatches. 
Such configurations are called 
wedge disclinations, which are angular defects}. Disclinations are observed in solids in situations where the rotational symmetry is violated at the level of the crystal lattice. Figure \ref{fig:discl:num} illustrates classical, simple examples of disclinations.

\begin{figure}[h]
\centering
\begin{tikzpicture}[scale=1]
\def \x {2.3}
\def \y {-2.8}
\def \xi {0}
\def \yi {0}
\def \a {7}
\begin{scope}[shift={(0,0)},scale=1]
  \node (label) at (\xi, \yi) {\includegraphics[height=6cm]{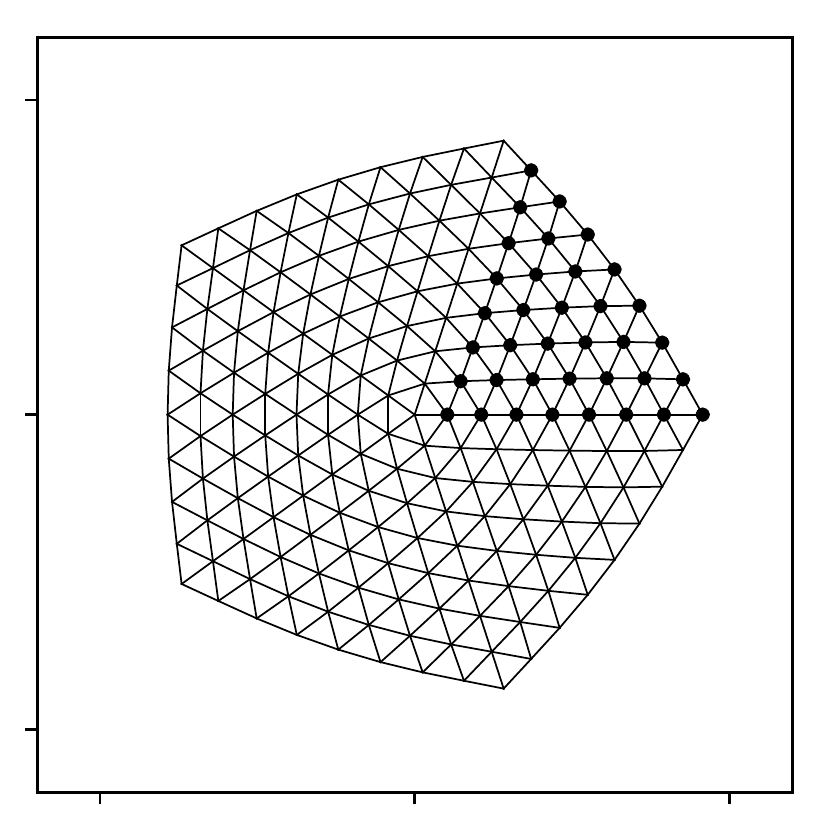}};
  \draw (-\x, \y) node[below] {$-1$};  
  \draw (0, \y) node[below] {$0$};
  \draw (\x, \y) node[below] {$1$};
  \draw (\y, -\x) node[left] {$-1$};  
  \draw (\y, 0) node[left] {$0$};
  \draw (\y, \x) node[left] {$1$};
  \draw (-2.6, \x) node[right] {$\phi = \tfrac{2\pi}5$};
\end{scope}
\begin{scope}[shift={(\a,0)},scale=1]
  \node (label) at (\xi, \yi){\includegraphics[height=6cm]{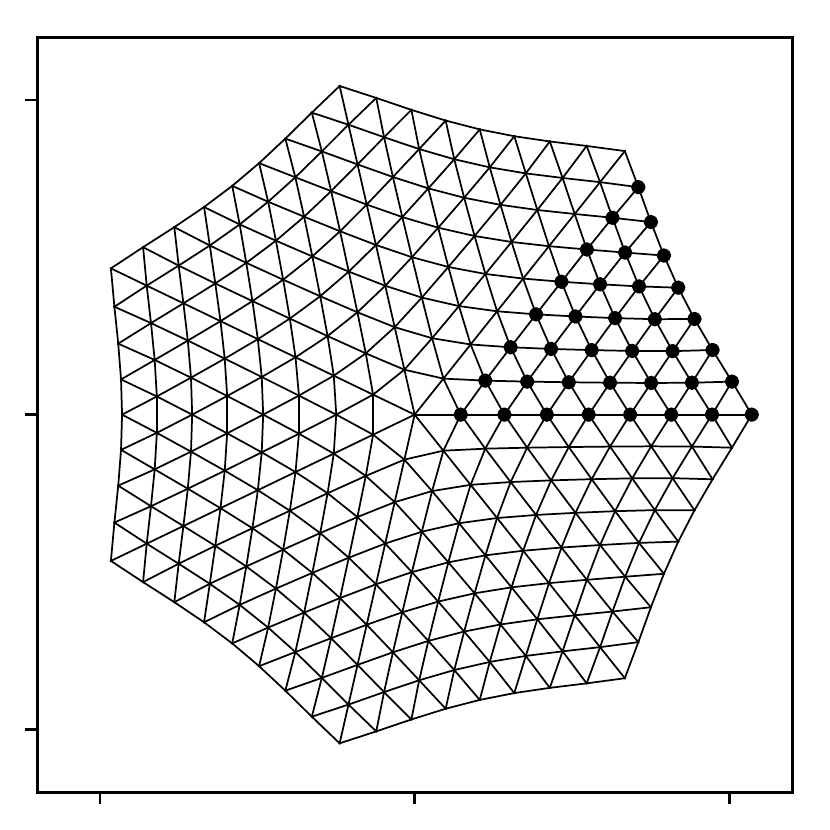}};
  \draw (-\x, \y) node[below] {$-1$};  
  \draw (0, \y) node[below] {$0$};
  \draw (\x, \y) node[below] {$1$};
  \draw (\y, -\x) node[left] {$-1$};  
  \draw (\y, 0) node[left] {$0$};
  \draw (\y, \x) node[left] {$1$};
  \draw (-2.6, \x) node[right] {$\phi = \tfrac{2\pi}7$};
\end{scope}
\end{tikzpicture}
\caption{Simple examples of two wedge disclinations in planar triangular lattices. 
The left figure shows a 5-type or positive disclination and the right figure a 7-type or negative disclination.
The atom in the center has a different number of bonds than the other atoms. $\phi$ is the angle between two such consecutive bonds. The lattice consists of rotated copies of the single wedge indicated by the thick dots. }
\label{fig:discl:num}
\end{figure}
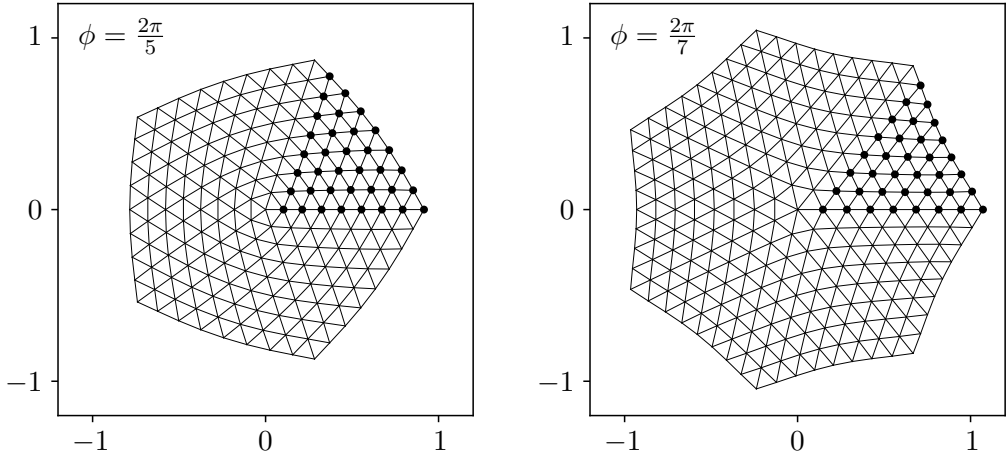

Historically, the existence of disclinations was predicted by Volterra alongside dislocations (translational defects) in a celebrated paper \cite{V07}. However, it was not until the late 1960s that disclinations saw a systematical investigation both from an experimental and theoretical perspective. 
First 
examples of disclinations over planar lattices have been discovered in superconductors and reported in 
\cite{TE68} and \cite{ET67}.
While a dislocation is
a singularity of the deformation field which may be described by a lattice-valued vector, called Burgers vector,
a disclination, as stated in
\cite{AEST68} (see also \cite{N67}),
  \textit{is characterized by a closure failure of rotation ... for a closed circuit round the disclination centre.} 
A continuum theory for disclinations in linearized elasticity has been first derived by de Wit in \cite{W68}
based on the idea of compatible elasticity and later elaborated by the same author in a series of   articles \cite{dW1}, \cite{dW2}, \cite{dW3}.
A comprehensive theory for disclinations (alongside dislocations) in non-linear elasticity has   been developed by Zubov \cite{Z97}.
%
%
We refer to \cite{ZA18} for 
more recent developments on unified approaches to treat dislocations and disclinations as well as for
a review of  mechanical models of disclinations.

%
%

Our work is inspired by a number of experimental observations
which
have revealed formation of disclinations in metallic structures 
under a variety of mechanical loading and forces, geometrical regimes and kinematical constraints.
%
Here, we elaborate on two such observations.

In austenite-to-martensite transformations, disclinations may emerge during the formation of rotated
(and constant-strain) shear-bands \cite{B}.
Such transformations are purely elastic and of the type solid-to-solid. They appear in a class of metals; in particular, in shape-memory alloys. Upon symmetry-break (typically triggered and driven by a negative temperature gradient), austenite, the high-symmetry and highly homogeneous crystal phase, turns into martensite, an anisotropic crystal phase with lower symmetry. The crystallographic lattice accommodates this phase change by forming a complicated  microstructure composed of a mixture of thin plates and needle-shaped regions exhibiting differently rotated copied of martensitic phases.
In a zero  stress-microstructure, martensite can be described by a
piecewise constant deformation gradient.  Constant-strain regions, that is, regions of constant  crystal orientation, are separated by sharp planar interfaces 
according to kinematical compatibility. This compatibility is called a rank-1 connection of the corresponding deformation tensors.
However, such configurations are ideal and typically
atomistic non-idealities such as dislocations and disclinations appear in large numbers.
%
%
Outstanding examples of disclinations   in martensite are represented by "nested" star-shaped geometries observed in 
$\textrm{Pb}_3(\textrm{VO}_4)_2$.
Experiments are described in
 \cite{KK91,MA80};
modeling work as well as numerical and exact constructions are described  in
\cite{PL13} and  \cite{CPL14} respectively, 
and mathematical theories are developed in
\cite{CDRZZ19}.
Significantly more complicated microstructures rich in defects are
described in \cite{ILTH17} for Ti-Nb-Al-based alloys. Here, martensite nucleates and evolves in the form of thin plates embedded in an austenitic lattice. The evolution of these plates is complex due to
plates colliding against surrounding structures,
undergoing further branching into additional martensitic sub-plates or
being reflected after hitting a grain boundary. Such evolution results in
self-similar patterns
resembling fractal structure
which are rich in 
disclinations and dislocations
\cite{ILTH17}.
The available models of such   microstructures
are essentially based on statistical analysis 
$\cite{IHM13}$
and probability
\cite{BCH15,CH20},
and thus the consistency with atomistic models remains elusive.
This is where we aim to contribute.
  
The second experimental observation is
the discovery of a superior
(and, yet to date, largely unexplained)
structural reinforcement mechanism which
has recently
spurred
scientific interest 
on the experimental
as well as on the theoretical investigation of kink formation
in certain classes of metal alloys.
Formation of kinks
consisting of approximately constant-strain
 bands 
accompanied by high rotational 
stretches of the lattice are observed
in classes of laminate "mille-feuille"
structures
under uniaxial compression 
\cite{HOIYMNK16,HMHYIOONK16}. 
In one of their most typical morphologies,  bands manifest themselves in the form of planar and sharp ridge-shaped regions which appear at various length-scales 
and are accompanied by  localized plastic stresses and formation of disclinations
\cite{LN15}. 
%
%
%
%
%
%
%
%
%
%
%
Kinks of various morphologies have been  
described in
\cite{I19}
where constructions of piecewise affine deformation maps based on the rank-1 connection rule
and incorporating angular mismatches
are presented.
While \cite{I19}'s analysis sheds light on the kinematics of the disclination-kinking mechanism for various morphologies of kinks, 
there is no available model based on atomistic descriptions which describes the energy of such disclination-kinking mechanisms.

A common aspect on both martensitic microstructures and kink formation is that
planar regions of approximately constant strains and constant-orientation lattices need to rotate in order to preserve the continuity of the deformation field
across their common border.
As a result, these materials \textit{necessarily} develop angular lattice misfits which are striking examples of wedge disclinations. This motivates the modeling assumption of this paper that wedge disclinations  are caused by     large (non-linear) rotational stretches in planar-confined geometries.

Our aim is to take the first step in the direction of 
a comprehensive variational theory 
that is suitable to simultaneously treat microscale and localized defects
and to  predict
their
effect on 
large (non-linear) elastic
and plastic deformations
including kinks and 
 shear-bands.
 By designing a simple, nearest-neighbor-type interaction mechanism, we construct a model 
which we apply to describe a single  disclination in a planar lattice and 
 which is at the same time potentially adaptable to describe
multi-disclination systems and to incorporate other lattice defects such as voids, dislocations and grain boundaries.

%
%



%

By pursuing this aim, we also fill a gap in the literature on atomistic modeling of planar lattice defects and the limit thereof as the lattice spacing tends to $0$. To identify this gap, we review related literature. 
First, in the framework of linearized elasticity, the formal asymptotic expansions in \cite{SN88} provide a continuum model for lattice defects. 
Second, in 
\cite{LazzaroniPalombaroSchloemerkemper13} 
two triangular lattices with different lattice spacings are attached together, which forces dislocations to form at the interface. The main result is a continuum limit of this model as the lattice spacing tends to 0.
Third, based on the discrete calculus of lattices constructed in \cite{ArizaOrtiz05}, the stress in a periodic crystal induced by parallel screw dislocations is computed in 
\cite{Ponsiglione07,
HudsonOrtner14,
HudsonOrtner15,
EhrlacherOrtnerShapeev16,
BraunBuzeOrtner19}.
These results provide quantitative estimates between the displacement in the lattice and the displacement computed from linear elasticity in the continuum counter part. In \cite{BuzeHudsonOrtner19} these techniques are extended to capture cracks. All these works deal with either localized or small deformation to the underlying two dimensional lattice, which is unfit for disclinations (see, e.g., the relatively large deformations in Figure \ref{fig:discl:num}).

For large deformations, the recent work in \cite{KM18} presents and anlyzes an atomistic model of a stretchable hexagonal lattice defined over a smooth manifold, in which the main result is the continuum limit in the form of $\Gamma$-convergence. Since the approach in \cite{KM18} is designed to describe  the energetics of highly distorted membranes, we follow a similar approach. However, the choice in \cite{KM18} that the number of bonds per atom is always $6$ makes the result not applicable to disclinations.

As in \cite{KM18}, we assign an energy $E_\e$ to a deformed lattice, where $\e$ is the lattice spacing. Two examples of such deformed lattices are illustrated by the dotted nodes in Figure \ref{fig:discl:num}. Given a deformed lattice, the energy $E_\e$ penalizes stretching and compression of the atomic bonds and changes in volume of the triangles formed by three neighboring atoms. To enforce the appearance of a $5$- or $7$-type disclination, we require the deformed lattice to satisfy a special boundary condition such that $5$ or $7$ rotated copies of it fit together such as in Figure \ref{fig:discl:num}.

%
%
%

Our main result, Theorem \ref{thm}, is the exact derivation of $E$, the macroscale energy obtained in the limit as $\e \to 0$, that is,
$$
E = \Gamma\hbox{-}\lim_{\varepsilon\to 0} E_{\varepsilon}
$$
in a suitable topology.



From the mathematical perspective, 
the interest of our analysis lies in the treatment of the special boundary condition.
%
%
Conceptually, this boundary conditions requires us to incorporate a pointwise constraint 
in the space of traces on $W^{1,p}$-Sobolev vector maps which
are characterized by means of a nonlocal norm. Our main contribution on this is Proposition \ref{prop:densy}.

 The paper is organized as follows. 
In Section $\ref{2003072140}$ we present the discrete mechanical model and the mathematical setting
required for the related variational analysis.
 In Section $\ref{s:t}$ we state and prove our main result, Theorem \ref{thm}, on the $\Gamma$-convergence of the lattice energy $E_\e$. In this proof, we postpone the proofs of technical lemmas to Section \ref{s:techs}, which includes the proof of Proposition \ref{prop:densy} on density in the space of admissible lattice displacements. 
 In Section $\ref{2003202308}$ we explore the physical implications of our $\Gamma$-convergence analysis in terms of energy and stress states of the minimizers of the continuum model.  Finally in Section $\ref{s:num}$ we present    several numerical realizations of both positive and negative disclinations.

\section{The lattice energy $E_\e$}\label{2003072140}

Here we define the atomic lattice energy $E_\e$ as briefly introduced in the introduction. We start with the kinematics. Inspired by Figure \ref{fig:discl:num}, we consider a two-dimensional model. This corresponds conceptually to the mid-section of a 3-dimensional body. Furthermore, we impose rotational symmetry so that we may confine the domain to a single wedge in Figure \ref{fig:discl:num} indicated by the black dots. The reference domain is given by the equilateral open triangle $\Omega \subset \R^2$ of size 1 with boundary $\Gamma = \cup_{i=1}^3 \Gamma_i$ as depicted in Figure \ref{fig:discl}. We take the reference positions of the atoms as a triangular lattice $\mathcal L_\e$, where $\mathcal L_\e$ is such that it fits on top of $\Omega$ as in Figure \ref{fig:discl}, i.e., the lattice spacing $\e \in (0,1)$ is such that $\frac1\e \in \N$ and $\mathcal L_\e$ is positioned such that each closed line segment $\Gamma_i$ fits on top of the atoms in one of the lattice directions. We denote by
\begin{equation} \label{Be}
  B^\varepsilon := \Big\{ \varepsilon R_\theta e_1 : \theta \in \frac\pi3 \Z \Big\}
\end{equation}
the set of the six outward-pointing bonds in $\mathcal L_\varepsilon$ from any lattice point, where $R_\theta \in SO(2)$ is the counter-clockwise rotation matrix by angle $\theta$. For later use, we sometimes interpret $\mathcal L_\e$ as a planar graph $\mathcal G_\varepsilon = (\mathcal V_\varepsilon, \mathcal E_\varepsilon, \mathcal T_\varepsilon)$, where $\mathcal V_\varepsilon$ is the set of all vertices $x_i$, $\mathcal E_\varepsilon$ is the set of all edges $e_{ij}$ between neighboring vertices $x_i, x_j \in \mathcal V_\varepsilon$, and $\mathcal T_\varepsilon$ is the set of all open triangles $T_{ijk} \in \Omega$ with sides given by the three edges $e_{ij}, e_{ik}, e_{jk} \in \mathcal E_\varepsilon$.

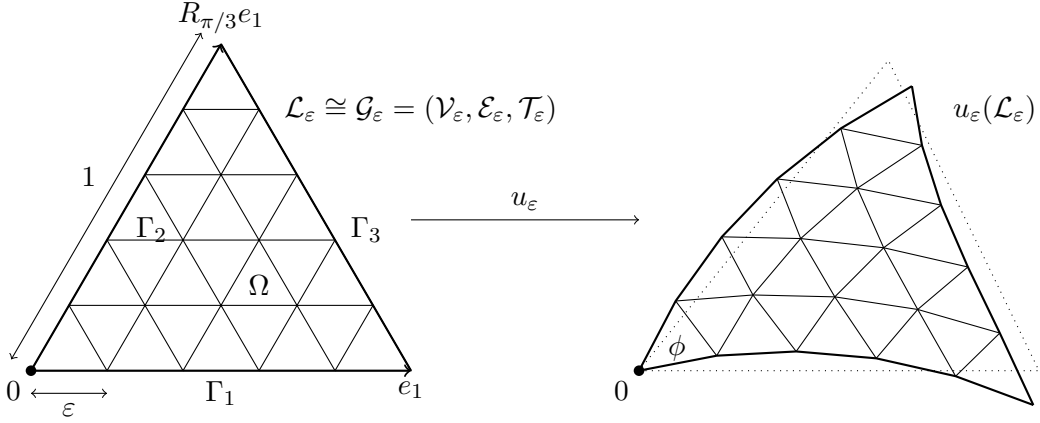
\begin{figure}[h!]
\centering
\begin{tikzpicture}[scale=1]
    \def \n {4} 
    \def \qthree {1.7321}
    \def \a {1.0525}
    \def \b {0.6562}
    \def \c {0.8229}

    \fill (0,0) circle (.07) node[anchor = north east]{$0$};
    \foreach \i in {0,...,\n} {
      \foreach \j in {\i,...,\n} {
        \def \x {\i/2 + \j/2}
        \def \y {\j*\qthree/2 - \i*\qthree/2}
        \begin{scope}[shift={(\x, \y)}] 
          \draw[thin] (0, 0) -- (1, 0) -- (0.5, \qthree/2) -- (0,0);
        \end{scope}
	  }
    }
    
    \begin{scope}[shift={(2*\n, 0)}]
      \fill (0,0) circle (.07) node[anchor = north east]{$0$};
      \draw[thin] (4.747, 0.499) -- (4.166, -0.073);
      \draw[thin] (4.327, 1.375) -- (3.667, 0.812) -- (3.121, 0.164);
      \draw[thin] (3.974, 2.196) -- (3.219, 1.630) -- (2.576, 0.979) -- (2.068, 0.255);
      \draw[thin] (3.723, 2.988) -- (2.873, 2.419) -- (2.127, 1.752) -- (1.499, 1.004) -- (1.022, 0.199);
      \draw[thick] (3.591, 3.773) -- (2.655, 3.212) -- (1.818, 2.542) -- (1.090, 1.775) -- (0.482, 0.924) -- (0.000, 0.000);
      \draw[thick] (5.188, -0.455) -- (4.166, -0.073) -- (3.121, 0.164) -- (2.068, 0.255) -- (1.022, 0.199) -- (0.000, 0.000);
      \draw[thin] (4.747, 0.499) -- (3.667, 0.812) -- (2.576, 0.979) -- (1.499, 1.004) -- (0.482, 0.924);
      \draw[thin] (4.327, 1.375) -- (3.219, 1.630) -- (2.127, 1.752) -- (1.090, 1.775);
      \draw[thin] (3.974, 2.196) -- (2.873, 2.419) -- (1.818, 2.542);
      \draw[thin] (3.723, 2.988) -- (2.655, 3.212);
      \draw[thick] (5.188, -0.455) -- (4.747, 0.499) -- (4.327, 1.375) -- (3.974, 2.196) -- (3.723, 2.988) -- (3.591, 3.773);
      \draw[thin] (4.166, -0.073) -- (3.667, 0.812) -- (3.219, 1.630) -- (2.873, 2.419) -- (2.655, 3.212);
      \draw[thin] (3.121, 0.164) -- (2.576, 0.979) -- (2.127, 1.752) -- (1.818, 2.542);
      \draw[thin] (2.068, 0.255) -- (1.499, 1.004) -- (1.090, 1.775);
      \draw[thin] (1.022, 0.199) -- (0.482, 0.924);
    \end{scope}
 
    \draw (\n + 1,0) -- (.5*\n + .5, \n*\qthree/2 + \qthree/2);
    \begin{scope}[shift={(2*\n, 0)}]
      \draw[dotted] (0,0) -- (\a*\n + \a,0) -- (\n*\b + \b, \c*\n + \c) -- cycle;
    \end{scope}
    \draw (\n - 1, \qthree/2) node[above]{$\Omega$};
    \draw (\n/2 + .5, 0) node[below]{$\Gamma_1$};
    \draw (\n/4 + .25, \qthree*\n/4 + \qthree/4) node[anchor = north west]{$\Gamma_2$};
    \draw (\n - \n/4 + 1.75, \qthree*\n/4 + \qthree/4) node[anchor = north east]{$\Gamma_3$};
    \draw (2*\n + .5, 0) node[above] {$\phi$};
    
    \draw[thick,->] (0, 0) -- (\n + 1, 0) node[below] {$e_1$};
    \draw[thick,->] (0, 0) -- (0.5*\n + 0.5, 0.5*\qthree*\n + 0.5*\qthree) node[above] {$R_{\pi/3} e_1$};
    \draw[thick] (\n + 1, 0) -- (0.5*\n + 0.5, 0.5*\qthree*\n + 0.5*\qthree);
    
    \draw[<->] (0, -0.3) -- (1, -0.3) node[midway, below]{$\varepsilon$};
    \begin{scope}[rotate = 60] 
      \draw[<->] (0, 0.3) -- (\n + 1, 0.3) node[midway, anchor = south east]{$1$};
    \end{scope}
    \draw (\n*.8, \n*\qthree/2) node[right]{$\mathcal L_\varepsilon \cong \mathcal G_\varepsilon = (\mathcal V_\varepsilon, \mathcal E_\varepsilon, \mathcal T_\varepsilon)$};
    \draw[->] (\n + 1, \n/2) -- (2*\n, \n/2) node[midway, above]{$u_\varepsilon$};
    \draw (3*\n, \n*\qthree/2) node[right]{$u_\varepsilon (\mathcal L_\varepsilon )$};
\end{tikzpicture} 
\caption{The reference lattice $\mathcal L_\varepsilon$ and an admissible deformation $u_\varepsilon$ (see \eqref{Ae}). The dotted triangle on the right is a visual aid to see that $u_\e$ satisfies the boundary condition.
} 
\label{fig:discl}
\end{figure}

The set $\mathcal A_{\varepsilon}$ of admissible displacements is given by
\begin{gather} \label{Ae}
  \mathcal A_\varepsilon = \{ u_\varepsilon : \mathcal L_\varepsilon \to \R^2 : u_\varepsilon ( R_{\pi/3} x ) = R_\phi u_\varepsilon (x) \ \text{ for all } x \in \Gamma_1 \cap \mathcal V_\varepsilon \},
\end{gather}
where $\phi \in \{ \frac{2\pi}5, \frac{2\pi}7 \}$ is the angle associated with a $5$- or $7$-type disclination (see Figure \ref{fig:discl:num}). Since $\mathcal V_\varepsilon$ is a finite set, the map $u_\e$ can be identified with a vector in $\R^{2 |\mathcal V_\varepsilon|}$. The boundary condition in \eqref{Ae} is such that the rotated copy $R_\phi u_\varepsilon (\mathcal L_\varepsilon )$ of $u_\varepsilon (\mathcal L_\varepsilon )$ fits seamlessly to $u_\varepsilon (\mathcal L_\varepsilon )$. Adding more rotated copies, we obtain a deformed lattice with rotational symmetry such as that in Figure \ref{fig:discl:num}, which has a $5$- or $7$-type disclination at the origin.  
\smallskip

Next we define the lattice energy on $\mathcal A_\varepsilon$ as inspired by \cite{KM18}. We start with a formal description. Given $u_\varepsilon \in \mathcal A_\varepsilon$, we set the lattice energy formally as
\begin{gather} \label{Ee:intro}
 \tilde E_\varepsilon (u_\varepsilon)
  = \varepsilon^2 \sum_{e_{ij} \in \mathcal E_\varepsilon} w(e_{ij}) \Phi \left( \frac{| u_\varepsilon(x_i) - u_\varepsilon(x_j) |}\varepsilon - 1 \right) 
    + \frac{ \varepsilon^2 }2 \sum_{T \in \mathcal T_\varepsilon} \Psi \Big( \sign(u_\e (T)) \frac{|u_\e (T)|}{|T|} \Big).
\end{gather}
where the potential
\begin{gather*}
  \Phi (r) := \frac1p | r|^p, \quad p \geq 2
\end{gather*}
penalizes atomic bonds which are not of length $\e$ (the parameter value $p=2$ corresponds to linear elasticity), 
the weight function
\begin{equation} \label{w}
  w : \mathcal E_\varepsilon \to \{\tfrac12, 1\},
  \qquad w(e_{ij}) = \left\{ \begin{aligned}
    & \tfrac12
    && \text{if } e_{ij} \subset \Gamma \\
    & 1
    && \text{otherwise}
  \end{aligned} \right.
\end{equation} 
counts the outer edges as half\footnote{we model edges as part of the volume of the medium around them.},
and the potential $\Psi$ penalizes change in volume of the triangles $T \in \mathcal T_\e$ (here, $|T|$ is the volume of $T$, and $|u_\e(T)|$ is the volume of the triangle $T$ after applying the displacement $u_\e$), especially if the volume of $T$ gets inverted under $u_\e$.
 Note that while the identity map minimizes $\tilde E_\e$, it is not in $\mathcal A_\e$ because of the boundary condition. Hence, the boundary condition enforces mechanical frustration. For $p=2$ and $\Psi \equiv 0$, Figure \ref{fig:discl:num} illustrates the deformed lattice of a local minimizer in $\mathcal A_\e$ of $\tilde E_\e$ which does not contain negative change in volume. In this paper we always assume $p\ge 2$ unless specified otherwise.

Since the term involving $\Psi$ is not derived from physical principles, we elaborate on this modelling choice. It is well-known that when $\Psi \equiv 0$, lattice energies with only nearest-neighbor interactions such as $\tilde E_\varepsilon$ do not penalize folding or any other negative change in volume, and as a result, the related continuum energy may not penalize compression. While our boundary condition in \eqref{Ae} is not standard, we show in Section \ref{s:num} by means of numerical simulations that folded patterns appear as local minimizers of $\tilde E_\e$ when $\Psi \equiv 0$. 

To penalize folding, we pose the following minimal requirements on $\Psi$:
\begin{itemize}
  \item[($\Psi 1$)] $\exists \, C > 0 \ \forall \, a, b \in \R : |\Psi(a) - \Psi(b)| \leq C \Big( |a|^{\tfrac p2 - 1} + |b|^{\tfrac p2 - 1} + 1 \Big) |a - b|$;
  \item[($\Psi 2$)]  $\Psi\geq 0$, and 
   $\Psi (a) = 0 \Longleftrightarrow a = 1$.
\end{itemize}
%
%
Condition ($\Psi 1$) is a continuity estimate, which implies both local Lipschitz continuity and $\frac p2$-growth.
Condition ($\Psi 2$) ensures that any change in volume is penalized.  A simple example of $\Psi$ is $\Psi(a) = |a-1|$. 
These conditions are more general than
those in \cite[Section 3.3]{KM18}.
%

\smallskip

The term in $\tilde E_\e$ related to $\Psi$ has an inconvenient form. We fix this by changing variables. The result is a rigorously defined energy functional $E_\e$, which we will use in the remainder of the paper. 

We change variables by describing the state $u_\e$ as the deformation of a displacement $U_\e$ defined on $\Omega$. Given $u_\varepsilon \in \mathcal A_\varepsilon$, we set $U_\varepsilon : \Omega \to \R^2$ as the piecewise linear extension of $u_\varepsilon$ to $\Omega$, i.e.,
\begin{multline*}
   U_\varepsilon ( \alpha x_i + \beta x_j + (1 - \alpha - \beta) x_k )
   := \alpha u_\varepsilon(x_i) + \beta u_\varepsilon(x_j) + (1 - \alpha - \beta) u_\varepsilon(x_k) \\
    \forall \, T_{ijk} \in \mathcal T_\varepsilon, 
            \ 0 \leq \alpha \leq 1, 
            \ 0 \leq \beta \leq 1 - \alpha.
\end{multline*} 
We note that
\begin{equation} \label{nUe:expl}
  \big[ \nabla U_\varepsilon (x) \big]^T e
  = \frac{ u_\varepsilon(x_\ell) - u_\varepsilon(x_m) }\varepsilon
    \quad \forall \, T_{ijk} \in \mathcal T_\varepsilon, 
            \ x \in T_{ijk},
\end{equation}
where $e \in B^1$ is any lattice bond of unit direction (see \eqref{Be}), and the indices $\ell, m \in \{i,j,k\}$ depend on $e$. In particular, $\nabla U_\e$ is constant on each $T \in \mathcal T_\varepsilon$,   
\begin{equation*} 
  | \det \nabla U_\e | = \frac{|U_\e (T)|}{|T|}
  \quad \text{on } T
\end{equation*}
is the relative change in volume of $T$ under $U_\e$, and the sign of the determinant determines whether the volume of $T$ is inverted under $U_\e$. From these observations, we define the second term in \eqref{Ee:intro} rigorously by
\begin{equation} \label{det:nUe}
  \Psi \Big( \sign(u_\e (T)) \frac{|u_\e (T)|}{|T|} \Big)
  := \Psi( \det \nabla U_\e )
  \quad \text{on } T.
\end{equation}

Next we rewrite $\mathcal A_\e$ and $\tilde E_\e$ in terms of $U_\e$. This yields
\begin{equation*} 
  \mathcal B_\varepsilon^\phi 
  := \big\{ U_\varepsilon : \overline \Omega \to \R^2 \mid U_\varepsilon \text{ is } \mathcal L_\varepsilon \text{-piecewise linear and }  
  \forall \, x \in \Gamma_1 : R_{\phi} U_\varepsilon(x) = U_\varepsilon ( R_{\pi/3} x) \big\}.
\end{equation*} 
Using \eqref{nUe:expl} and \eqref{det:nUe}, we rewrite \eqref{Ee:intro} as
\begin{align} \notag
  \frac{\sqrt 3}2 \tilde E_\e(u_\e)
  &= \frac{\sqrt 3}2 \bigg( \varepsilon^2 \sum_{T \in \mathcal T_\varepsilon} \frac12 \sum_{e \in B^1} \frac1{|T|} \int_T \Phi \left( \big| e \nabla U_\varepsilon (x) \big| - 1 \right) \, dx \\\notag
  &\qquad \qquad + \frac{ \varepsilon^2 }2 \sum_{T \in \mathcal T_\varepsilon} \frac1{|T|} \int_T \Psi ( \det \nabla U_\e (x) ) \, dx \bigg) \\\notag
  &= \int_\Omega \sum_{e \in B^1} \Phi \left( \big| e \nabla U_\varepsilon (x) \big| - 1 \right) + \Psi ( \det \nabla U_\e (x) ) \, dx  \\\label{Ee:Te:form}
  &= \int_\Omega W (\nabla U_\varepsilon (x) ) \, dx,
\end{align} 
where
\begin{equation} \label{W:KM17}
  W(A) := \sum_{e \in B^1} \Phi \left( \big| e A \big| - 1 \right) + \Psi( \det A ).
\end{equation}   
In the computation above, the weight function $w(e_{ij})$ turns into the factor $\frac12$ due to the fact that each edge in the interior of $\Omega$ borders two triangles in $\mathcal T_\varepsilon$. Eq.\ \eqref{Ee:Te:form} motivates us to define
\begin{equation} \label{Ee:rig}
  E_\e : \mathcal B_\varepsilon^\phi \to [0, \infty),
  \qquad E_\varepsilon (U_\varepsilon) 
  := \int_\Omega W (\nabla U_\varepsilon (x) ) \, dx.
\end{equation}

\begin{rem}[Properties of $\mathcal B_\varepsilon^\phi$]\label{2005171600}
We note that $\mathcal B_\varepsilon^\phi \subset \Lip (\overline \Omega; \R^2)$ and that any $U_\varepsilon \in \mathcal B_\varepsilon^\phi$ satisfies $U_\varepsilon(0) = 0$. 
While $\mathcal B_\e^\phi$ does not contain a subspace of constants, it does contain a subspace of linear maps $U_\e(x) = Ax$. A possible choice for $A \in \R^{2 \times 2}$ is the one that satisfies $A e_1 = e_1$ and $A (R_{\pi/3} e_1) = R_\phi e_1$. 
\end{rem}

\section{Continuum limit}
\label{s:t}

Having introduced the lattice energy $E_\e$ on the triangular lattice, we are now in a position to discuss the continuum limit as $\e \to 0$.
%
%
To keep track of the asymptotic behavior of minima and minimizers of $E_\e$
we characterize the continuum model with $\Gamma$-convergence \cite{DalMaso93}.
 
Let $p \geq 2$. The domain of the continuum energy is
\begin{equation*}
  W_\phi^{1,p}(\Omega; \R^2) := \bigg\{ U \in W^{1,p} (\Omega; \R^2) : R_{\phi} U_\varepsilon(x) = U_\varepsilon ( R_{\pi/3} x) \text{ for a.e.\ } x \in \Gamma_1 \bigg\}.
\end{equation*}
Observe that $W_\phi^{1,p}(\Omega; \R^2)$ is linear  (and, in particular, convex),
and non-empty since
$ \mathcal B_\e^\phi \cup W_0^{1,p}(\Omega; \R^2) \subset  W_\phi^{1,p}(\Omega; \R^2)$.
Moreover, if
$U\in W_\phi^{1,p}(\Omega; \R^2)$, there holds $RU\in W_\phi^{1,p}(\Omega; \R^2)$
for any $R\in SO(2)$.
Thanks to the properties of traces (see Lemma \ref{lem:Trace} below), we have that $W_\phi^{1,p}(\Omega; \R^2)$ is strongly closed and, therefore, weakly closed as well thanks to convexity.

The continuum energy functional  is given by
\begin{equation} \label{E:QW}
  E : L^p( \Omega; \R^2 ) \to [0,\infty], \qquad  
  E (U)
  = \left\{ \begin{aligned}
    & \int_\Omega QW(\nabla U)
    && U \in W_\phi^{1,p}(\Omega; \R^2) \\
    & \infty
    && U \notin W_\phi^{1,p}(\Omega; \R^2),
  \end{aligned} \right.
\end{equation} 
where $W$ is as in \eqref{W:KM17} and $QW$ is the quasiconvex  envelop  of $W$ defined by 
\begin{equation}\label{2004262117}
  QW (A) := \sup \{ f(A) : V \leq W, \ f \text{ quasiconvex} \},
\end{equation}
where $f$ being quasiconvex means that 
$f : \R^{2 \times 2} \to \R$ is 
Borel measurable, locally bounded and  satisfies
\begin{equation*}
  f(A) \leq \frac1{|\omega|} \int_{\omega} f(A + \nabla \vartheta(x)) \, dx
\end{equation*}
for any bounded open set $\omega \subset \R^2$, 
any $A \in \R^{2 \times 2}$ and any $\vartheta \in W^{1,\infty}_0 (\omega,\mathbb{R}^2)$.

\begin{thm} \label{thm}
For $\phi \in \{ \frac{2\pi}5, \frac{2\pi}7 \}$ and $2 \leq p < \infty$, $E_\varepsilon$ (see \eqref{Ee:rig}) $\Gamma$-converges as $\varepsilon \to 0$ to $E$ in the strong $L^p(\Omega)$ topology.
\end{thm}

We prove Theorem \ref{thm} in Section \ref{s:t:pf}. Since the $\varepsilon$-dependence of $E_\varepsilon$ appears only in its domain $\mathcal B_\phi^\e$, proving a $\Gamma$-limit result reduces to proving a relaxation result, i.e., finding the lower semi-continuous envelope of 
\begin{equation*} 
  U \mapsto \int_\Omega W(\nabla U)
\end{equation*}
in the right functional framework. There is a large literature on such relaxation problems; in Section \ref{s:t:classical} we cite the relevant classical theory. While this theory gives a useful roadmap for proving Theorem \ref{thm}, it does not capture Theorem \ref{thm} because of the periodic boundary condition in $\mathcal B_\varepsilon^\phi$. Therefore, in Section \ref{s:t:pf} we give the skeleton of the proof of Theorem \ref{thm} based on the classical theory, and identify the missing steps as technical lemmas which we prove in Section \ref{s:techs}.

\subsection{Classical relaxation result on $\mathcal E$}
\label{s:t:classical}
We review some classical relaxation results as preparation for proving Theorem \ref{thm}. All theorem references below in this section refer to Dacorogna's book \cite{Dac08}. Another relevant reference is \cite{AcerbiFusco84}.

We recall that $\Omega$ is a bounded Lipschitz domain. 
Here and in what follows we adopt the
Frobenius norm for matrices.
Let $\cW : \R^{2 \times 2} \to \R$ satisfy
\begin{itemize}
  \item[($W 1$)] $p$\emph{-growth}. $\exists \, C, C' > 0 \ \forall \, A \in \R^{2 \times 2} :
        \dfrac{1}{C}|A|^p - C \leq \cW(A) \leq C' (|A|^p + 1)$;
  \item[($W 2$)] \emph{Continuity estimate}. $\exists \, C > 0 \ \forall \, A, B \in \R^{2 \times 2} : \dfrac{|\cW(A) - \cW(B)|}{|A - B|} \leq C (|A|^{p-1} + |B|^{p-1} + 1)$.
\end{itemize}

Note that ($W 1$) includes a uniform bound from below, and that ($W 2$) provides a local Lipschitz estimate. We set
\begin{equation*} 
  \cE : L^p( \Omega; \R^2 ) \to \R, \qquad  
  \cE (U)
  = \left\{ \begin{aligned}
    & \int_\Omega \cW(\nabla U)
    && U \in W^{1,p}(\Omega) \\
    & \infty
    && U \notin W^{1,p}(\Omega).
  \end{aligned} \right.
\end{equation*}
Here and henceforth, we remove the range $\R^2$ from the notation of the function space if there is no danger for confusion.

As preparation, we cite Theorem 6.9 for the alternative characterization of $(\ref{2004262117})$ as the quasiconvexification of $\cW$ given by
\begin{equation}\label{2004262119}
  Q\cW (A) = \inf\left \{ |\omega|^{-1}\int_{\omega}\cW(A+\nabla\vartheta) : 
  \vartheta\in W^{1,\infty}_0(\omega,\mathbb{R}^2) \right\},
\end{equation}
where $\omega$ is a subset of $ \R^2$
with $|\partial\omega|=0$ (i.e., $\partial\omega$ has zero two-dimensional volume).

Since $\cW$ satisfies ($W 1$), Theorem 9.1 implies that for all $U \in W^{1,p}(\Omega)$ there exists a sequence $(U_{k})_k \subset W_0^{1,p}(\Omega) + \{U\}$ such that 
\begin{eqnarray*}
U_{k}\to U \quad \text{in } L^p(\Omega)
\end{eqnarray*}
and
\begin{eqnarray}\label{1803051356}
\mathcal E (U_k) =
\int_{\Omega} \cW(\nabla U_{k}(x)) \, dx \to \int_{\Omega} Q \cW (\nabla U(x)) \, dx \quad \text{ as } k\to\infty.
\end{eqnarray}

By the two properties of $\cW$, it follows from Theorem 6.9 and Theorem 5.3(iv) that $Q \cW$ is continuous. Then, Theorem 1.13 implies that $U \mapsto \int Q\cW (\nabla U)$ is sequentially weakly lower semicontinuous in $W^{1,p}(\Omega)$. In particular, for any $(U_{k})_k \subset W^{1,p}(\Omega)$ converging weakly in $W^{1,p}(\Omega)$ to some $U$, we have that
\begin{eqnarray}\label{1803051355}
\liminf_{k \to \infty} \mathcal E (U_k) 
\geq \liminf_{k \to \infty} \int_{\Omega} Q\cW(\nabla U_{k}(x)) \, dx 
\geq \int_{\Omega} Q\cW (\nabla U(x)) \, dx.
\end{eqnarray}

\subsection{Proof of Theorem \ref{thm}}
\label{s:t:pf}

The proof below relies on the following three statements which we make precise and prove in Section \ref{s:techs}:
\begin{itemize}
  \item Poincar\'e Inequality holds on $W^{1,p}_\phi (\Omega)$ (Lemma \ref{l:Poincare});
  \item $W$ satisfies Properties ($W1$) and ($W2$) defined in Section \ref{s:t:classical} (Lemma \ref{l:prop:W});
  \item $\cB_\phi^\varepsilon$ is dense in $W^{1,p}_\phi (\Omega)$ (Proposition \ref{prop:densy}).
\end{itemize}
The proof of 
Theorem $\ref{thm}$ follows from matching a lower bound with an upper bound, which is  the standard method for computing $\Gamma$-limits \cite{DalMaso93}.
To identify the set where the limit functional is finite, we first
investigate the (equi-)compactness of minimizing sequences.


\begin{proof}[Proof of Theorem \ref{thm}]
\textbf{Compactness}. By the lower bound in ($W1$), 
\begin{equation*}
  E_\varepsilon (U_\varepsilon) \geq \frac1C \| \nabla U_\varepsilon \|_{L^p(\Omega)}^p - C
\end{equation*}
for some constants $C > 0$ independent of $\e$ and $U_\e$, and thus any finite-energy sequence $(\nabla U_\varepsilon)_\varepsilon$ is bounded in $L^p(\Omega)$. By the Poincar\'e Inequality (Lemma \ref{l:Poincare}), we then infer that $U_\varepsilon$ is bounded in $W^{1,p}(\Omega)$, and thus strongly convergent (along a subsequence) in $L^p(\Omega)$.

\textbf{$\Gamma$-liminf}. Since we can focus on finite-energy sequences $U_\varepsilon \to U$ in $L^p(\Omega)$, the compactness statement implies that $U_\varepsilon \weakto U$ in $W^{1,p}(\Omega)$ as $\varepsilon \to 0$. 
Since $(U_\varepsilon) \subset W^{1,p}_\phi (\Omega)$, which is a closed subspace of $W^{1,p} (\Omega)$, we also have $U \in W^{1,p}_\phi (\Omega)$. Hence, by applying \eqref{1803051355}, we obtain the required $\Gamma$-liminf estimate.
 
\textbf{$\Gamma$-limsup}. Thanks to \eqref{1803051356} it suffices to find, for any $U \in W_\phi^{1,p} (\Omega)$, a sequence $(U_\varepsilon)_\varepsilon \subset \mathcal B_\phi^\varepsilon$ such that
\begin{equation} \label{p:limp:cond}
  U_\varepsilon \xto{\varepsilon \to 0} U \ \text{ in } L^p (\Omega) 
  \quad \text{and} \quad 
   E_\e (U_\varepsilon) \xto{\varepsilon \to 0} \int_{\Omega} W(\nabla U)dx =: \cE (U).
\end{equation}
To prove \eqref{p:limp:cond}, we infer from  Proposition \ref{prop:densy} and $E_\e (U_\varepsilon) = \cE (U_\varepsilon)$ that it is enough to show that $\cE$ is continuous in $W^{1,p}(\Omega)$ at $U$. The continuity of $\cE$ follows by ($W2$), Lemma \ref{l:Poincare} and H\"older's inequality from
\begin{align*}
  | \cE(U) - \cE(V) | 
  &\leq \int_\Omega | W(\nabla U) - W(\nabla V) | \, dx \\
  &\leq C \int_\Omega \big( | \nabla U |^{p-1} + | \nabla V |^{p-1} + 1 \big) | \nabla U - \nabla V | \, dx \\
  &\leq C \left( \| U \|_{W^{1,p}(\Omega)}^{p-1} + \| V \|_{W^{1,p}(\Omega)}^{p-1} + 1 \right) \| U - V \|_{W^{1,p}(\Omega)},
\end{align*}
where $V \in W^{1,p} (\Omega)$ is arbitrary.
\end{proof}

\section{Technical steps in the proof of Theorem \ref{thm}}
\label{s:techs}

Here we state rigorously and prove the three statements mentioned at the start of Section \ref{s:t:pf}. We start with the Poincar\'e Inequality.

\begin{lem}[Poincar\'e Inequality on $W^{1,p}_\phi (\Omega)$]\label{l:Poincare}
There exists $C > 0$ such that for all $U \in W^{1,p}_\phi (\Omega)$ it holds that
\begin{equation*}
  \| U \|_{ L^p (\Omega) } \leq C \| \nabla U \|_{ L^p (\Omega) }.
\end{equation*}
\end{lem} 
\begin{proof}
We follow a standard proof by contradiction. Assuming that there exists $(U_n) \subset W^{1,p}_\phi (\Omega)$ such that
\begin{equation*}
  1 = \| U_n \|_{ L^p (\Omega) } > n \| \nabla U_n \|_{ L^p (\Omega) },
\end{equation*}
we obtain by compactness that $U_n$ converges along a subsequence (not relabelled) to $U$ strongly in $L^p(\Omega)$ and weakly in $W^{1,p}_\phi (\Omega)$. Hence, $\| U \|_{ L^p (\Omega) } = 1$ and $\| \nabla U \|_{ L^p (\Omega) } = 0$. Since $\Omega$ is connected, $U \equiv V_0$ for some constant and non-zero vector $V_0 \in \R^2$. However, for a.e.~$x \in \Gamma_1$,
\begin{equation*}
  R_\phi U (x) - U ( R_{\pi/3} x ) 
  = (R_\phi - I)V_0
  \neq 0.  
\end{equation*}
Hence, $U \notin W^{1,p}_\phi (\Omega)$, which completes the proof.
\end{proof}

\subsection{Properties of $W$}

\begin{lem} \label{l:prop:W}
$W$ defined in \eqref{W:KM17} satisfies Properties ($W1$) and ($W2$) defined in Section \ref{s:t:classical}.
\end{lem}

\begin{proof}
If $\Psi \equiv 0$, then ($W1$) is obvious since $\Phi(r) = \frac1p {|}r{|}^p$. Then, since $\Psi \geq 0$ by ($\Psi 2$),
the lower bound in ($W1$) is immediate. The upper bound follows by ($\Psi 1$) from
\begin{align*}
    \Psi(\det A) 
   &= \Psi(\det A) - \Psi(\det I) \\
   &\leq C \Big( |\det A|^{\tfrac p2 - 1} + 2 \Big) |\det A - 1| \\
   &\leq C' (|A|^{p-2} + 2) (|A|^2 + 1). 
 \end{align*}
($W2$) follows by ($\Psi 1$) from 
\begin{align*}
  &|W(A) - W(B)| \\
  &\leq \sum_{e \in B^1} \big| \Phi \left( | eA | - 1 \right) - \Phi \left( | eB | - 1 \right) \big| + \big| \Psi( \det A ) - \Psi( \det B ) \big| \\
  &\leq \sum_{e \in B^1} C_p \big( | e A |^{p-1} + | e B |^{p-1} + 1 \big) \big| | e A | - | e B | \big| \\
  &\qquad + C \Big(|\det A|^{\tfrac p2 - 1} + |\det B|^{\tfrac p2 - 1} + 1 \Big) \big| \det A - \det B \big|  \\
  & \leq C' \big( |A|^{p-1} + |B|^{p-1} + 1 \big) |A - B| + 
  C'' \big( |A|^{p-2} + |B|^{p-2} + 1 \big) (|A|+|B|) |A - B|
\\  & \leq C''' \big( |A|^{p-1} + |B|^{p-1} + 1 \big) |A - B|.
\end{align*}
\end{proof}

For later use (although not necessary for the proof of 
 Theorem \ref{thm}) we elaborate on the rigidity properties of the energy density $W$. The rigidity estimate is a direct consequence of Assumptions (W1), (W2).
Lengthy computations are postponed to the Appendix.

We start with introducing the singular value decomposition 
\begin{eqnarray}\label{2005071602}
A= P_1 \Sigma P_2, \quad \text{where } \Sigma := \begin{bmatrix}
     \sigma_1 & 0 \\ 0 & \sigma_2
   \end{bmatrix} \textrm{ and } P_1,P_2 \in O(2).
\end{eqnarray}
In (\ref{2005071602}), $0 \leq \sigma_1 \leq \sigma_2$ are the ordered singular values of $A$, that is,
the square roots of the eigenvalues of the matrix $AA^T$. 
We also recall the definition of the distance function
\begin{eqnarray*}
\dist^p(A,SO(2))=\inf_{Q\in SO(2)}|A-Q|^p.
\end{eqnarray*}

\begin{lem} \label{l:W:rigid}
For $W$ defined in  (\ref{W:KM17}), there exists $c_2>0$ such that for all $A \in \R^{2 \times 2}$
\begin{eqnarray}\label{2005171442}
W(A)\geq c_2 \dist^p(A,SO(2)).
\end{eqnarray}
 
\end{lem}


\begin{proof}
Let $A =P_1\Sigma P_2 $ be the singular decomposition as in \eqref{2005071602}. We split two cases;  $\det A \ge 0$ and $\det A < 0$. We start with the first case. From the definition of $W$ we get
\begin{eqnarray}\label{2005071448}
W(A) \geq  \frac{1}{p} \sum_{e \in B_1}
  \bigl | |Ae| - 1  \bigr | ^p = \frac{1}{p}\sum_{e \in B_1}  \bigl | |\Sigma P_2 e| - 1  \bigr | ^p
=  \frac{1}{p}\sum_{k=0}^5   \bigl | |\Sigma R_{\theta + k\pi/3} e_1| - 1   \bigr |^p, 
\end{eqnarray}
where $e_1$ is the unit vector, $R$ is the rotation matrix, and $\theta \in [0, \pi/3)$ is fixed by $P_2$. 
We first consider the case $p=2$.
Thanks to Lemma \ref{2005071702}
we have
\begin{eqnarray*}
14 \sum_{k=0}^5   \bigl ||\Sigma R_{\theta + k\pi/3} e_1| - 1 \bigr |^2
   \geq (\sigma_1 - 1)^2 + (\sigma_2 - 1)^2.
\end{eqnarray*}
Recalling  the  well-known relation
$$
  (\sigma_1 - 1)^2 + (\sigma_2 - 1)^2 = \dist^2(A,O(2))
$$ 
and noting that $\det A \geq 0$ implies $\dist (A, O(2)) = \dist (A, SO(2))$, we obtain \eqref{2005171442}. 
%
%
For the case $p>2$, applying Jensen's inequality in \eqref{2005071448} yields 
\begin{eqnarray*} 
W(A)  
   \geq\frac{6^{1-\frac{p}{2}}}{p} \Bigl  |  \sum_{k=0}^5 (|\Sigma R_{\theta + k\pi/3} e_1| - 1)^2  \Bigr |^{\frac{p}{2}}.
\end{eqnarray*} 
Then, (\ref{2005171442}) follows from the argument above.


\medskip

We continue with the second case, $\det A < 0$. By $(W1)$ in Section \ref{s:t:classical}, there exist $M > 0$ independent of $A$ such that $W(A) \geq \frac1M (|A|^p - M)$. We separate two cases: 
   \begin{enumerate}
    \item If $|A| \leq 2M$, then $\det A \geq - 2M^2$, and thus 
    \[ W(A) 
       \geq \Psi (\det A) 
       \geq \min_{[-2M^2, 0]} \Psi, \]
    which is positive by $(\Psi 2)$. On the other hand, for the right-hand side of \eqref{2005171442}, we obtain
    \[ \dist^p (A, SO(2)) \leq \max_{|B| \leq 2M} \dist^p (B, SO(2)) < \infty. \] 
    Hence, there exists $c_2 > 0$ independent of $A$ such that \eqref{2005171442} holds.
    \item If $|A| > 2M$, we note from the triangle inequality 
    \[ \dist(A, SO(2)) \leq |A| + \dist(0, SO(2)) = |A| + \sqrt 2 \]
    that
    \[ W(A) 
      \geq \frac1M |A|^p - M
      \geq c_2 (|A| + \sqrt 2)^p
      \geq c_2 \dist^p(A, SO(2)) \]
      for some $c_2 > 0$ independent of $A$.
  \end{enumerate}
\end{proof}

\subsection{Density of $\cB_\phi^\varepsilon$ in $W^{1,p}_\phi (\Omega)$}

\begin{prop}[Density of $\cB_\phi^\varepsilon$ in $W^{1,p}_\phi (\Omega)$] \label{prop:densy}
For all $U \in W^{1,p}_\phi (\Omega)$ there exists $U_\varepsilon \in B_\phi^\varepsilon$ parametrized by $\frac1\varepsilon \in \N$ such that $\| U - U_\varepsilon \|_{W^{1,p} (\Omega)} \to 0$ as $\varepsilon \to 0$.
\end{prop}
Before giving the proof of Proposition \ref{prop:densy} at the end of this section, we first outline the idea of the proof, and then establish some technical lemmas.

In order to explain the idea of the proof, we first recall two classical density results in the following lemma. To state it, we define $\Lip_\phi (\overline \Omega) := \Lip (\overline \Omega) \cap  W^{1,p}_\phi (\Omega)$.

\begin{lem}[Density of $B_\phi^\varepsilon$] \label{lem:ET} 
 
For any $U : \Omega \to \R^2$ with either $U \in W^{1,p}_0 (\Omega)$ or $U \in C^1(\Omega) \cap \Lip_\phi (\overline \Omega)$ there exists $U_\varepsilon \in B_\phi^\varepsilon$ parametrized by $\frac1\varepsilon \in \N$ such that $\| U - U_\varepsilon \|_{W^{1,p} (\Omega)} \to 0$ as $\varepsilon \to 0$.
\end{lem}

\begin{proof}
This is a standard result in numerical analysis; see e.g.~\cite[Chap.~X, Prop.~2.1, 2.6 and 2.9]{EkelandTemam99}. We give the details of the proof to show how our boundary condition fits in.

For $U \in W^{1,p}_0 (\Omega)$, it is not restrictive by density to assume that $U \in C^\infty_c (\Omega)$. Then, setting $U_\varepsilon (x_i) := U(x_i)$ with piecewise affine continuation, it is obvious that $U_\varepsilon \in B_\phi^\varepsilon$ and that $\nabla U_\varepsilon \to \nabla U$ uniformly as $\varepsilon \to 0$.

For $U \in C^1(\Omega) \cap \Lip_\phi (\overline \Omega)$, we set again $U_\varepsilon (x_i) := U(x_i)$ with piecewise affine continuation, and define $V_\varepsilon := U_\varepsilon - U$. Since $U \in C (\overline \Omega)$, we have $\| V_\varepsilon \|_\infty \to 0$, and thus $\| V_\varepsilon \|_{L^p(\Omega)} \to 0$. For the gradient, we split $\Omega = \Omega_\delta \cup N_\delta$, where the disjoint sets $\Omega_\delta$ and $N_\delta$  are such that  $\overline{ \Omega_\delta } \subset \Omega$ and $N_\delta$ has volume that vanishes as $\delta \to 0$. Then, 
\begin{equation*}
  \| \nabla V_\varepsilon \|_{L^p(\Omega)}^p
  = \| \nabla V_\varepsilon \|_{L^p(\Omega_\delta)}^p + \| \nabla V_\varepsilon \|_{L^p(N_\delta)}^p.
\end{equation*}
First, since $\nabla U \in C (\overline{ \Omega_\delta })$, we have by the argument above that $\| \nabla V_\varepsilon \|_{C (\overline{ \Omega_\delta })} \to 0$, and thus $\| \nabla V_\varepsilon \|_{L^p(\Omega_\delta)}^p \to 0$ as $\e \to 0$. Second, since $U \in \Lip_\phi (\overline \Omega)$, $\| \nabla U_\e \|_\infty \leq \| \nabla U \|_\infty < \infty$, and thus $\| \nabla V_\varepsilon \|_{L^p(N_\delta)}^p \leq C |N_\delta|$ uniformly in $\e$. Hence, by taking $\e$ small enough with respect to $\delta$ and $\delta \to 0$, we conclude that $\| \nabla V_\varepsilon \|_{L^p(\Omega)} \to 0$ as $\e \to 0$.
\end{proof}

Thanks to Lemma \ref{lem:ET}, the proof of Proposition \ref{prop:densy} narrows down to constructing a decomposition $U = U_1 + U_2 + U_3$ where
\begin{equation} \label{Ui:props}
  \left\{ \begin{aligned}
    & U_1 \in C^1({\Omega}) \cap \Lip_\phi (\overline{\Omega}), \\
    & U_2 \in W_0^{1,p}(\Omega), \\
    & \| U_3 \|_{W^{1,p}(\Omega)} \text{ is small.} 
  \end{aligned} \right.
\end{equation}
Indeed, if such a decomposition exists, then Lemma \ref{lem:ET} provides approximations in $B_\phi^\varepsilon$ of $U_1$ and $U_2$, and $U_3$ can simply be approximated by $0$. The difficulty in constructing such a decomposition is in finding a $U_1$ for which $U_1 |_\Gamma$ is sufficiently close to $U |_\Gamma$. Approximation by convolution does not work directly since $U_1$ has to satisfy the boundary condition in $\Lip_\phi (\overline \Omega)$. Instead, we use the Trace Theorem to approximate $U |_\Gamma$ in an appropriate function space on $\Gamma$. This approximation is based on convolution, but care is needed because of the boundary condition, the corners of $\Gamma$ and the fact that the norm of the function space on $\Gamma$ is nonlocal.  
\medskip

Next, we prepare for proving Proposition \ref{prop:densy} by citing a Trace Theorem (Lemma \ref{lem:Trace}) and proving a density result on $\Gamma$ (Lemma \ref{lem:apx:on:Gamma}). To avoid technical difficulties with the corners in $\Gamma$, we first transform $\Omega$ to the unit disc $\Dd = \{ x \in \R^2 : |x| < 1 \}$. With this aim, let $\varphi : \overline{\Omega} \to \overline{\Dd}$ be a related transformation (see Figure \ref{fig:varphi}) such that 
\begin{itemize}
  \item $\varphi$ is bi-Lipschitz;
  \item $\varphi \in C^1(\Omega)$;
  \item $\varphi(\Gamma_i) = \gamma_i$ for $i = 1,2,3$, where $\gamma_i \subset \partial \Dd$ are given, in terms of the polar angle coordinate, by 
  \begin{equation*}
  \gamma_1 := [0, \tfrac{2\pi}3], \quad 
  \gamma_2 := [-\tfrac{2\pi}3, 0], \quad 
  \gamma_3 := [\tfrac{2\pi}3, \tfrac{4\pi}3]; 
\end{equation*}
  \item if $U \in W_\phi^{1,p}(\Omega)$, then
  \begin{equation*}
    U \circ \varphi^{-1} \in W_\phi^{1,p}(\Dd) := \{ V \in W^{1,p}(\Dd) : V(-\theta) = R_\phi V(\theta) \: \text{ for a.e.\ } \, 0 < \theta < \tfrac{2\pi}3 \}.
  \end{equation*}
\end{itemize} 

\begin{figure}[h!]
\centering
\begin{tikzpicture}[scale=1]
    \def \n {3} 
    \def \qthree {1.7321}
    \def \a {1.0525}
    \def \b {0.6562}
    \def \c {0.8229}
    
    \draw[->] (-.5,0) -- (\n + 1.5,0);
    \draw[->] (0,-.5) -- (0, \n*\qthree/2 + \qthree/2 + .5);
    \draw[thick] (0,0) -- (\n + 1,0) node[midway, below]{$\Gamma_1$} node[below]{$1$} -- (.5*\n + .5, \n*\qthree/2 + \qthree/2) node[midway, right]{$\Gamma_3$} -- cycle node[midway, left]{$\Gamma_2$};
    \draw (\n - 1, 2*\qthree/3) node {$\Omega$};
    
    \draw[thick, ->] (\n+1.5, \qthree) --++ (2.5,0) node[midway, above] {$\varphi$};
    
    \begin{scope}[shift={(10,\qthree)},scale=1] 
      \draw[->] (-2.5,0) -- (2.5,0);
      \draw[->] (0,-2.5) -- (0,2.5);
      \draw[thick] (0,0) circle (2);
      \draw (.5,-.5) node {$\Dd$};
      \draw (-2,0) node[anchor = south east] {$\gamma_3$};
      \draw (.7,0) arc (0:30:.7);
      
      \draw[thick] (1.8,0) --++ (.4,0);
      \begin{scope}[rotate = 15] 
        \draw (1,0) node {$\theta$};
      \end{scope}
      \begin{scope}[rotate = 30] 
        \draw[thin] (0,0) -- (2,0) node[midway, above]{$1$};
      \end{scope}
      \begin{scope}[rotate = 60] 
        \draw (2,0) node[above] {$\gamma_1$};
      \end{scope}
      \begin{scope}[rotate = -60] 
        \draw (2,0) node[below] {$\gamma_2$};
      \end{scope}
      \begin{scope}[rotate = 120] 
        \draw[dotted] (0,0) -- (2,0);
        \draw[thick] (1.8,0) --++ (.4,0);
      \end{scope}
      \begin{scope}[rotate = 210] 
      \end{scope}
      \begin{scope}[rotate = 240] 
        \draw[dotted] (0,0) -- (2,0);
        \draw[thick] (1.8,0) --++ (.4,0);
      \end{scope}
    \end{scope}
\end{tikzpicture} 
\caption{The deformation $\varphi$ and the related sections of the boundaries of $\Omega$ and $\Dd$.} 
\label{fig:varphi} 
\end{figure}
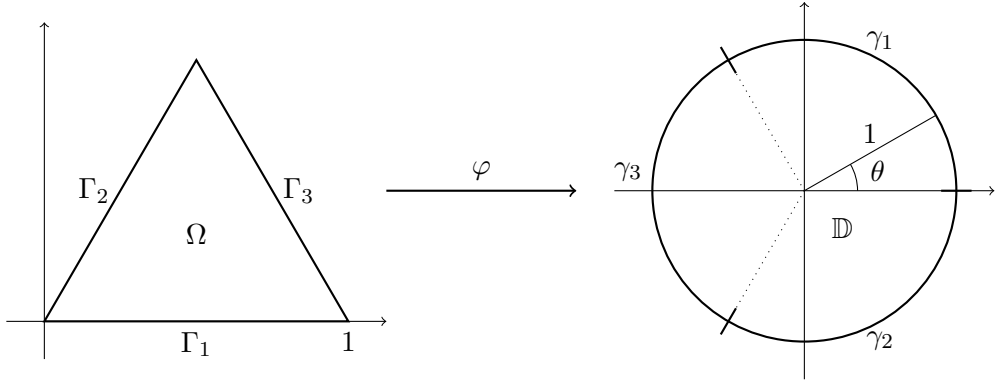

Above and in the following, we will often identify the unit circle $\Ss := \partial \Dd$ with the periodic interval $[0, 2\pi)$. We also adopt the convention that a subscript $\phi$ in a function space indicates the boundary condition. We note that there are constants $0 < c, C$ such that for all $U \in W_\phi^{1,p}(\Omega)$
\begin{equation*}
  c\| U \|_{W^{1,p}(\Omega)} \leq \| U \circ \varphi^{-1} \|_{W^{1,p}(\Dd)} \leq C\| U \|_{W^{1,p}(\Omega)}.
\end{equation*}
Hence, it is sufficient to construct the decomposition of $U$ after transforming it to $W_\phi^{1,p}(\Dd)$.

To cite the Trace Theorem, we fix $s = 1 - \frac1p$ and recall the usual norm of the fractional Sobolev space $W^{s,p}(\Ss)$:
\begin{align*}
  \| U \|_{W^{s,p}(\Ss)}^p
  &:=  \| U \|_{L^p(\Ss)}^p + [ U ]_{W^{s,p}(\Ss)}^p, \\
  [ U ]_{W^{s,p}(\Ss)}^p
  &:= \int_0^{2\pi} \int_0^{2\pi} \frac{|U(\theta) - U(\rho)|^p}{|\theta-\rho|_\Ss^p} \, d\rho d\theta,
  \qquad |\theta|_\Ss = \min_{k \in \Z} |\theta - 2\pi k|.
\end{align*} 

\begin{lem}[Trace {\cite[Thm.~1.I]{Gagliardo57}}] \label{lem:Trace}
There exists a $C > 0$ such that for all $U \in W^{1,p}(\Dd)$ 
\begin{equation*} 
  \big\| U |_\Ss \big\|_{W^{s,p}(\Ss)}
  \leq C \| U \|_{W^{1,p}(\Dd)}.
\end{equation*}
Conversely, there exists a $C > 0$ such that for all $f \in W^{s,p}(\Ss)$ there exists a $U \in W^{1,p}(\Dd)$ with $U |_\Ss = f$ such that
\begin{equation*} 
  \| U \|_{W^{1,p}(\Dd)}
  \leq C \| f \|_{W^{s,p}(\Ss)}.
\end{equation*}
\end{lem}

In order to prove a density result in $W_\phi^{s,p}(\Ss)$, we first recall two technical lemmas:

\begin{lem}[{\cite{DiNezzaPalatucciValdinoci12}, Lem.~5.2}] \label{lem:Hhalf:even}
If $f \in W^{s,p}(0,\pi)$, then the even extension $\tilde f$ satisfies
\begin{equation*}
  \| \tilde f \|_{W^{s,p}(-\pi,\pi)} \leq C \| f \|_{W^{s,p}(0,\pi)}
\end{equation*}
for some universal constant $C > 0$.
\end{lem}

\begin{lem}[Continuity of the translation operator] \label{lem:trl}
The translation operator $\tau_a f (\theta) := f (\theta + a)$ is continuous in the strong $W^{s,p}(\mathbb{S})$ topology, that is, for all $f \in W^{s,p}(\mathbb{S})$,
$$
\lim_{a \to 0}\|\tau_af-f \|_{W^{s,p}(\mathbb{S})}=0. 
$$
\end{lem}

\begin{proof}
The proof is standard; we give a sketch. The statement is obvious for $f \in C^1(\Ss)$. For general $f \in W^{s,p}(\mathbb{S})$, it suffices to approximate it by $\psi \in C^1(\Ss)$, and note that 
\[ \| \tau_a \psi - \tau_a f \|_{W^{s,p}(\mathbb{S})}
    = \| \tau_a (\psi - f) \|_{W^{s,p}(\mathbb{S})} 
    = \| \psi - f \|_{W^{s,p}(\mathbb{S})}.\]
\end{proof}

\begin{lem}[Approximation on $\Ss$] \label{lem:apx:on:Gamma}
$C_\phi^\infty (\Ss)$ is dense in $W_\phi^{s,p} (\Ss)$.
\end{lem}

\begin{proof}
Let $f \in W_\phi^{s,p} (\Ss)$. We split $f = f_1 + f_2 + f_3$ with $f_i \in W^{s,p} (0, 2\pi)$ such that
\begin{equation*}
  f_1 = \left\{ \begin{array}{ll}
     f &\text{on } (0, \tfrac{5\pi}6) \\
 0   &\text{on } (\tfrac{7\pi}6, 2\pi),
  \end{array} \right.
  \qquad
  f_2 = \left\{ \begin{array}{ll}
      0   &\text{on } (0, \tfrac{5\pi}6) \\ 
     f &\text{on } (\tfrac{7\pi}6, 2\pi).
  \end{array} \right.
\end{equation*}
Note that $f_3 \in W_\phi^{s,p} (\Ss)$ with $\supp f_3 \subset [\tfrac{5\pi}6, \tfrac{7\pi}6] \subset \subset \gamma_3$, and that it is not directly clear that $f_1, f_2 \in W^{s,p} (\Ss)$. We will construct approximating sequences $(\psi_{12}^k)_k, (\psi_3^k)_k \subset C^\infty(\Ss)$ of $f_1 + f_2$ and $f_3$ respectively such that $\psi^k := \psi_{12}^k + \psi_3^k \in C_\phi^\infty(\Ss)$. By construction of $f_3$, $\psi_3^k := \eta_k * f_3$ is a suitable approximating sequence, where $\eta_k$ is the usual mollifier.

To construct $\psi_{12}^k$, we first show that $f_1, f_2 \in W^{s,p} (\Ss)$. Since by construction $f_1, f_2 \in W^{s,p} (0, 2\pi)$, it is sufficient to show that $f_1, f_2 \in W^{s,p} (-r, r)$ for some $r > 0$. We start with $f_1$. For any $\theta \in \Ss$ and any parameter $0 \leq \alpha \leq \pi$, let
\begin{equation*}
  f_1^\alpha (\theta) := \left\{ \begin{array}{ll}
     R_\alpha f_1(-\theta) & -\pi < \theta < 0 \\
     f_1(\theta) & 0 < \theta < \pi.
  \end{array} \right.
\end{equation*}
Clearly, $f_1^\alpha \in W^{s,p}(0, \pi) \cap W^{s,p}(\pi, 2\pi)$ for any $\alpha$. Since $f_1^\phi (\theta) = f (\theta)$ for a.e.~$\theta \in (-\tfrac{5\pi}6, \tfrac{5\pi}6)$ and $f \in W_\phi^{s,p} (\Ss)$, we also have $f_1^\phi \in W^{s,p} (-\tfrac{5\pi}6, \tfrac{5\pi}6)$. 
Since $f_1^0$ is the even extension of $f_1 |_{(0,\pi)}$, we obtain from Lemma \ref{lem:Hhalf:even} that $f_1^0 \in W^{s,p} (-\pi, \pi)$. Hence, for the odd extension $f_1^\pi$, we find from the linear relation
\begin{equation*}
  f_1^\phi = \frac{I + R_\phi}2 f_1^0 + \frac{I - R_\phi}2 f_1^\pi
\end{equation*}
that $f_1^\pi \in W^{s,p} (-\tfrac{5\pi}6, \tfrac{5\pi}6)$. Finally,
\begin{equation*}
  f_1 = \frac{f_1^0 + f_1^\pi}2 \in W^{s,p} (-\tfrac{5\pi}6, \tfrac{5\pi}6 ).
\end{equation*}
The proof of $f_2 \in W^{s,p} (-\tfrac{5\pi}6, \tfrac{5\pi}6)$ is analogous.

Finally we construct the approximating sequences $\psi_{12}^k$. Care is needed to ensure that $\psi_{12}^k$ satisfies the boundary condition. With this aim, we first approximate $f_1$ and $f_2$ by the translations $\tau_a f_1$ and $\tau_{-a} f_2$ with $0 < a < \frac\pi6$; see Lemma \ref{lem:trl}. Note that $\tau_a f_1 + \tau_{-a} f_2 \in W_\phi^{s,p} (\Ss)$ and $(\tau_a f_1 + \tau_{-a} f_2)|_{(-a,a)} \equiv 0$. Then, we define $\psi_{12}^k := \eta_k * (\tau_{a_k} f_1 + \tau_{-{a_k}} f_2)$ for some $a_k \to 0$ such that $\supp \eta_k \subset (-a_k, a_k)$, and note that $\psi_{12}^k \in C_\phi^\infty(\Ss)$. This completes the proof.
\end{proof}

We are ready to prove Proposition \ref{prop:densy}.

\begin{proof}[Proof of Proposition \ref{prop:densy}]
\textit{Step 1: decomposition of $U$.}
Let $\delta > 0$ and $U \in W^{1,p}_\phi (\Omega)$ be given. Set $\tilde U := U \circ \varphi^{-1} \in W^{1,p}_\phi (\Dd)$ and $\tilde f := \tilde U |_\Ss$. Then, by Lemma \ref{lem:Trace} we have $\tilde f \in W_\phi^{s,p}(\Ss)$. By Lemma \ref{lem:apx:on:Gamma} we find $\tilde \psi \in C_\phi^\infty (\Ss)$ with $\| \tilde \psi - \tilde f \|_{W^{s,p} (\Ss)} < \delta$. By Lemma \ref{lem:Trace}, there exists $\tilde U_3 \in W^{1,p}_\phi (\Dd)$ with $\tilde U_3 |_\Ss = \tilde f - \tilde \psi$ and 
\[ 
  \| \tilde U_3 \|_{W^{1,p} (\Dd)} 
  \leq C \| \tilde \psi - \tilde f \|_{W^{s,p}(\Ss)} 
  \leq C \delta
\]
for some $C > 0$ independent of $\delta$.

Next, we take $\tilde U_1 \in C_\phi^\infty (\Dd)$ as the harmonic extension of $\tilde \psi$. Then, translating back to $\Omega$, we set $U_i := \tilde U_i \circ \varphi$ for $i = 1, 3$. Taking $U_2 := U - U_1 - U_3$, we observe that $U_i$ satisfy \eqref{Ui:props} for $i=1,2,3$.
\smallskip

\textit{Step 2: Construction of $U_\varepsilon$.}
By Lemma \ref{lem:ET} we find sequences $(U_1^\varepsilon)_\varepsilon, (U_2^\varepsilon)_\varepsilon$ parametrized by $\frac1\varepsilon \in \N$ such that $U_i^\varepsilon \in B_\phi^\varepsilon$ and 
\begin{equation*}
  \| U_i - U_i^\varepsilon \|_{W^{1,p}(\Omega)}
  \xto{ \varepsilon \to 0 } 0
\end{equation*}
for $i = 1,2$. Hence, setting $U_\varepsilon := U_1^\varepsilon + U_2^\varepsilon$ and taking $\varepsilon$ small enough with respect to $\delta$, we obtain
\begin{equation*}
  \| U - U_\varepsilon \|_{W^{1,p}(\Omega)}
  \leq \| U_1 - U_1^\varepsilon \|_{W^{1,p}(\Omega)}
  + \| U_2 - U_2^\varepsilon \|_{W^{1,p}(\Omega)} + \| U_3 \|_{W^{1,p} (\Omega)}
  \leq C \delta
\end{equation*}
for some $C$ independent of $U_i$, $\delta$ or $\varepsilon$. Since $\delta > 0$ is arbitrary, we conclude that $\| U - U_\varepsilon \|_{W^{1,p}(\Omega)} \to 0$ as $\e \to 0$.
\end{proof}


\section{Physical interpretation of Theorem \ref{thm}}\label{2003202308}

We show that,
despite the relaxation process implied by  $\Gamma$-convergence, the equilibrium solution of the model are necessarily stressed. This is a consequence of the rotational boundary conditions
incorporated into $W^{1,p}_{\phi}(\Omega)$ and of 
the finite penalization to 
folding
imposed by
$\Psi$. 

\begin{prop}\label{p:Egtr0}
\begin{eqnarray}
\min_{L^p(\Omega)} E(U) >0.
\end{eqnarray}
\end{prop}
The proof of Proposition \ref{p:Egtr0} follows the lines of  \cite[Sec. 5]{KM18}. For the readers' convenience, we display the main steps.

\begin{proof}
First, we show 
\begin{equation} \label{pf:QWF:LB}
  QW(F)\geq c\dist^p(F,SO(2)),
\end{equation}
for some $c>0$.
From Lemma \ref{l:W:rigid}, we have
 $W(F)\geq c_2 \dist^p(F,SO(2))$ for every $F\in\mathbb{R}^{2\times 2}$.
Now, thanks to (\ref{2004262119}), for every $\varepsilon>0$ and every smooth open $\omega \subset \R^2$ with $|\omega|=1$, there exists a map $\vartheta_{\varepsilon}\in W^{1,\infty}_0(\omega,\mathbb{R}^{2})$
(in fact even $C^{\infty}_c(\omega,\mathbb{R}^{2})$)
  such that
    $$
QW(F)
\ge \int_{\omega} W (F + \nabla\vartheta_{\varepsilon} (x))dx-\varepsilon
\ge c_2\int_{\omega}\dist^p(F+\nabla\vartheta_{\varepsilon}(x), SO(2))dx-\varepsilon.
 $$
 We now invoke   Rigidity Theorem    3.1 \cite{FJM02} 
 (which applies for $p\ge 2$, see  \cite[Sec. 2.4]{CS06}) 
 yielding the existence of a constant $c > 0$ and a matrix 
 $R\in SO(2)$ such that
\begin{eqnarray}\label{2003262137}
c_2\int_{\omega}\dist^p(F+\nabla\vartheta_{\varepsilon}(x), SO(2))dx
\geq c \int_{\omega}|F+\nabla\vartheta_{\varepsilon}(x)-R|^p dx.
\end{eqnarray}
Now, interpreting $| \cdot |^p : \R^{2 \times 2} \to \R$ as a convex function on $\R^4$, we obtain
\begin{eqnarray}\label{2003262139}
|(F-R)+\nabla\vartheta_{\varepsilon}(x)|^{p}\geq |F-R|^{p} + p|F-R|^{p-2}(F-R):\nabla\vartheta_{\varepsilon}(x)
\end{eqnarray}
for all $x \in \omega$.
By applying (\ref{2003262139}) to  (\ref{2003262137}) we have
\begin{eqnarray*}
\int_{\omega}\dist^p(F+\nabla\vartheta_{\varepsilon},SO(2))dx \geq 
\int_{\omega}|F-R|^p dx,
\end{eqnarray*}
 because the integral of $(F-R):\nabla\vartheta_{\varepsilon}$ vanishes as $\vartheta_{\varepsilon}$ has zero boundary datum. Putting our estimates together and minimising over $R$, we get
$$
 QW(F)\ge c\int_{\omega}\dist^p(F,SO(2))dx-\varepsilon
$$ 
and the desired result in \eqref{pf:QWF:LB} follows by sending $\varepsilon\to 0$. 

Applying \eqref{pf:QWF:LB} we obtain
\begin{eqnarray*}
\min_{W^{1,p}_{\phi}(\Omega)} \mathcal{E}(U)
= \mathcal{E}(\ov U)
\geq \int_{\Omega}\dist^p(\nabla \ov U,SO(2)) dx.
\end{eqnarray*}
It is left to prove that the right-hand side is positive. Suppose instead that it is $0$.
Then, $\nabla \ov U \in SO(2)$ a.e.\ in $\Omega$. Hence, $\nabla \ov U$ is a constant rotation (see \cite{BJ87}). 
This contradicts with $\ov U \in W^{1,p}_{\phi}(\Omega)$.
\end{proof}

\section{Numerical computations}
\label{s:num}

In this section we explore numerically several energy wells of $E_\varepsilon$ and inspect whether the obtained minimizers satisfies $\det \nabla U_\varepsilon > 0$. 
We do this for the simplified setting given by $\Psi \equiv 0$, which does not penalize folded patterns. Fortunately, this simplified setting turns out to be stable enough to reproduce configurations with $\det \nabla U_\varepsilon > 0$ (see, e.g., Figure \ref{fig:discl:num}) as \textit{local} minima.
Yet, local minimizers are of limited use to our theoretical result Theorem \ref{thm}, because $\Gamma$-convergence only guarantees the convergence of \textit{global} minimizers in the limit $\varepsilon\to 0$.
Instead, we show by direct inspection of the energy values at the computed local minimizers that there is a good agreement with available continuum models for disclinations for decreasing values of $\varepsilon$.
We compute local minimizers by employing Newton's method for different initial conditions. In all computations below, Newton's method converges quadratically.

\subsection{The limit $\varepsilon \to 0$}
\label{s:num:e0}

Here we verify and quantify some of the computations in \cite{SN88}, in which disclinations satisfying $\det \nabla U_\varepsilon > 0$ are obtained. The only difference with \cite{SN88} is in the definition of $w$ in \eqref{w}; in \cite{SN88} $w \equiv 1$ is taken instead.

For both $\phi = \frac{2 \pi}5$ and $\phi = \frac{2 \pi}7$ we chose the initial condition such that the local minimizer of $E_\varepsilon$ satisfies $\det \nabla U_\varepsilon > 0$. With this aim, for $\e = 2^{-1}$ we took as initial condition one of the linear deformations $U_\e(x) = Ax$ in $B_\phi^\e$ with $\det A = 1$ (see Remark \ref{2005171600}). For the subsequent values $\e = 2^{-2}, 2^{-3},\ldots,2^{-8}$, 
 we constructed the initial condition from the minimizer obtained for the previous value of $\varepsilon$ by linear interpolation. 
The resulting local minimizers $\overline U_\e$ turn out to satisfy $\det \nabla \overline U_\varepsilon > 0$. For $\e = 2^{-3}$, $\overline U_\e$ is illustrated in Figure \ref{fig:discl:num}. The energy values $e_\e := E_\e (\overline U_\e)$ are plotted in Figure \ref{fig:Evals}. The   behavior of $e_\e$ as $\e \to 0$ is qualitatively similar to the computations of 
\cite{SN88} which are based on a model in linearized elasticity.

Next, 
we test the convergence of $e_\e$ as $\e \to 0$. Indeed, since $\overline U_\e$ need not be global minimizers, the $\Gamma$-convergence result in Theorem \ref{thm} does not guarantee convergence. To test the convergence, we impose the power law ansatz 
\begin{equation} \label{ee}
  e_\e \sim C_1 - C_2 2^p,
\end{equation}
where the constants $C_1, C_2, p > 0$ need to be fitted from the data. Without computing $C_1$ and $C_2$, we obtain the exponent $p$ of the power law numerically from
\begin{equation} \label{pe}
p_\e := \frac1{\log 2} \log \frac{ e_{4\e} - e_{2\e} }{ e_{2\e} - e_\e }.  
\end{equation}
Table \ref{tab:pvals} lists the values of $p_\e$. Since the values of $p_\e$ are positive and increasing with $-\log \e$, $e_\e$
 seems to converge as $\e \to 0$
 implying a power law behavior.


\begin{table}[h!]                    
\centering                           
\begin{tabular}{ccc} \toprule   
 & \multicolumn{2}{c}{$\phi$} \\ \cmidrule(r){2-3}                                 
$\e$ & $\tfrac{2\pi}5$ & $\tfrac{2\pi}7$ \\ \midrule                   
$2^{-3}$ & 1.552 & 1.276 \\
$2^{-4}$ & 1.712 & 1.488 \\
$2^{-5}$ & 1.812 & 1.595 \\
$2^{-6}$ & 1.866 & 1.645 \\
$2^{-7}$ & 1.898 & 1.671 \\
$2^{-8}$ & 1.918 & 1.686 \\
\bottomrule
\end{tabular}     
\medskip                   
\caption{The values of $p_\e$ (see \eqref{pe} as computed from the numerical data to test the power law in the ansatz in \eqref{ee}).}
\label{tab:pvals}           
\end{table}

\begin{figure}[h]
\centering
\begin{tikzpicture}[scale=1]
\def \x {4.2}
\def \y {2.75}
\def \xi {0}
\def \yi {0}
\def \a {7}

\node (label) at (\xi, \yi) {\includegraphics[height=6cm]{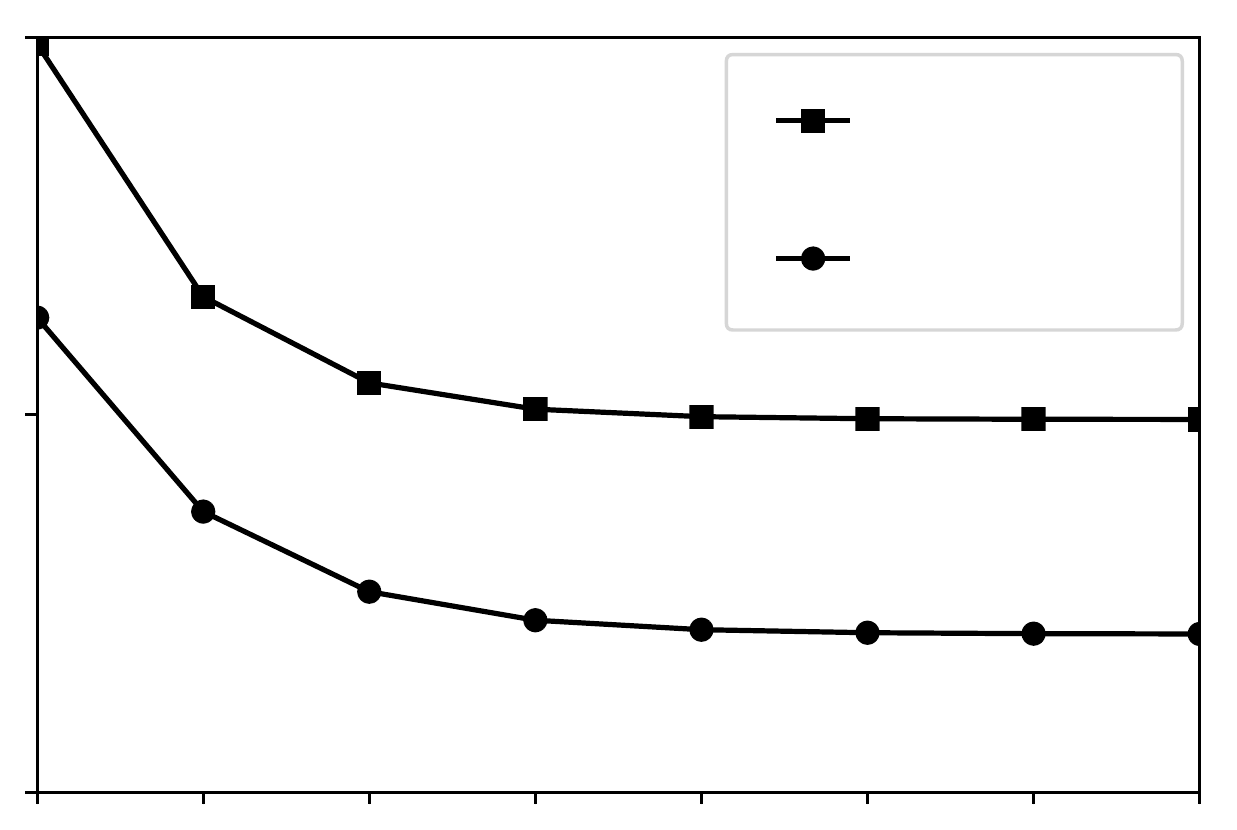}};
\draw (-\x, -\y) node[below] {$2^{-1}$}; 
\draw (-3*\x/7, -\y) node[below] {$2^{-3}$}; 
\draw (\x/7, -\y) node[below] {$2^{-5}$};
\draw (5*\x/7, -\y) node[below] {$2^{-7}$};
\draw (\x, -\y) node[anchor = south west] {$\e$}; 
\draw (-\x, -\y) node[left] {$1$};
\draw (-\x, 0) node[left] {$2$};
\draw (-\x, \y) node[left] {$3$} node[anchor = south west] {$e_\e \times 10^3$};
\draw (2, 2.13) node[right] {$\phi = \tfrac{2\pi}5$};
\draw (2, 1.13) node[right] {$\phi = \tfrac{2\pi}7$};

\end{tikzpicture}
\caption{The energy values $e_\e = E_\e (\overline U_\e)$ obtained from the simulations. Here, $\overline U_\e$ is verified to satisfy $\det \nabla \overline U_\e > 0$.}
\label{fig:Evals}
\end{figure}

\subsection{Folded configurations}
\label{s:num:fold}

In the next simulations we explore other local minimizers of the energy by starting Newton's method from other initial conditions. In particular, since we take $\Psi \equiv 0$, there is no penalisation on $\det \nabla U_\varepsilon \le 0$, and thus local minimizers with folded patterns may occur. 
%
%
%
%

With this aim, we set $\varepsilon = 2^{-2}$ and fold the reference lattice as illustrated in Figure \ref{fig:fold}. This folding procedure is done such that the following property used in the boundary condition in \eqref{Ae} 
\[ \Gamma_2 \cap \mathcal V_\varepsilon = \{ R_{\pi / 3} x : x \in \Gamma_1 \cap \mathcal V_\varepsilon \} \]
is conserved during the folding process. After folding, we construct the initial condition for Newton's method by deforming the folded reference domain by the linear map 
defined in Remark \ref{2005171600}.

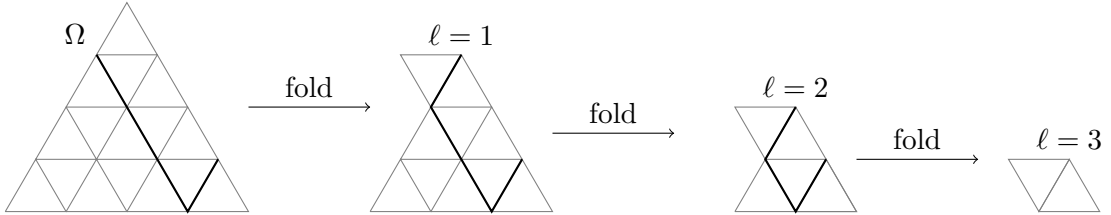
\begin{figure}[h!]
\centering
\begin{tikzpicture}[scale=.8]
    \def \n {4} 
    \def \qthree {1.7321}
    \def \a {1}
    \def \c {0.866}
    \def \b {.5}
    \def \r {.1}
    
    \begin{scope}[shift={(\a*7, 0)}]
    \foreach \i in {0,...,3} {
      \foreach \j in {\i,...,3} {
        \def \x {\i*\a + \j*\b - \i*\b}
        \def \y {\j*\c - \i*\c}
        \begin{scope}[shift={(\x, \y)}] 
          \draw[gray] (0, 0) -- (\a, 0) -- (\b, \c) -- (0,0);
        \end{scope}
	  }
    }
    
    \draw[thick] (3*\b, 3*\c) node[anchor = south east]{$\Omega$} -- (3*\a, 0) -- (3*\a + \b,  \c);
    \draw[->] (4*\a, 2*\c) -- (6*\a, 2*\c) node[midway, above]{fold};
    \end{scope}
    
    \begin{scope}[shift={(\a*13, 0)}]
    \foreach \i in {0,...,2} {
      \foreach \j in {\i,...,2} {
        \def \x {\i*\a + \j*\b - \i*\b}
        \def \y {\j*\c - \i*\c}
        \begin{scope}[shift={(\x, \y)}] 
          \draw[gray] (0, 0) -- (\a, 0) -- (\b, \c) -- (0,0);
        \end{scope}
	  }
    }
    \begin{scope}[shift={(2*\b, 2*\c)}, rotate = 60] 
      \draw[gray] (0, 0) -- (\a, 0) -- (\b, \c) -- cycle;
    \end{scope}
    
    \draw[thick] (3*\b, 3*\c) -- (2*\b, 2*\c) -- (2*\a, 0) -- (2*\a + \b,  \c);
    \draw (3*\b, 3*\c) node[above] {$\ell = 1$};
    \draw[->] (3*\a, 1.5*\c) -- (5*\a, 1.5*\c) node[midway, above]{fold};
    \end{scope}
    
    \begin{scope}[shift={(\a*19, 0)}]
    \foreach \i in {0,1} {
      \foreach \j in {\i,1} {
        \def \x {\i*\a + \j*\b - \i*\b}
        \def \y {\j*\c - \i*\c}
        \begin{scope}[shift={(\x, \y)}] 
          \draw[gray] (0, 0) -- (\a, 0) -- (\b, \c) -- (0,0);
        \end{scope}
	  }
    }
    \begin{scope}[shift={(\b, \c)}, rotate = 60] 
      \draw[gray] (0, 0) -- (\a, 0) -- (\b, \c) -- cycle;
    \end{scope}
    
    \draw[thick] (2*\b, 2*\c) -- (\b, \c) -- (\a, 0) -- (\a + \b,  \c);
    \draw (2*\b, 2*\c) node[above] {$\ell = 2$};
    \draw[->] (2*\a, \c) -- (4*\a, \c) node[midway, above]{fold};
    \end{scope}
    
    \begin{scope}[shift={(\a*24, 0)}]
    \begin{scope}[shift={(0, 0)}, rotate = 60] 
      \draw[gray] (0, 0) -- (\a, 0) -- (\b, \c) -- cycle;
    \end{scope}
    \draw[gray] (0, 0) -- (\a, 0) -- (\b, \c) -- cycle;
    \draw (\b, \c) node[above] {$\ell = 3$};
    \end{scope}
\end{tikzpicture} 
\caption{The folding procedure of the reference domain to construct the initial conditions. Lines along which we fold are drawn in boldface.} 
\label{fig:fold}
\end{figure}
 
  Figures \ref{fig:5fold} and \ref{fig:7fold} illustrate the local minimizers obtained for each folded pattern shown in Figure \ref{fig:fold} for $\phi = \frac{2\pi}5$ and $\phi = \frac{2\pi}7$ respectively. Figure \ref{fig:Efolds} shows the related energy values. It is clear that the folded local minimizers have lower energy. In fact, it seems that the energy values decrease in an affine manner with the number of folds, and that, after linear extrapolation, $0$ energy would be obtained around 4 folds.

The observation that $E_\e$ has folded patterns as local minima with low energy is not unexpected from the expression of the continuum energy $E$. Indeed, $\Psi=0$
implies that the minimization
problem of $E$ lacks rigidity; Lemma $\ref{l:W:rigid}$
and Proposition $\ref{p:Egtr0}$ are not valid anymore. In fact,
if we substitute $\Psi=0$
into (\ref{W:KM17}),
we obtain (see Lemma \ref{l:QW0}) $QW(0)=0$, which implies that $0 \in L^p(\Omega)$ is a global minimizer of $E$.
%
%
%
This is consistent with the observation in Figure \ref{fig:Efolds}, where the energy values decay with the number of folds.


\begin{figure}[h]
\centering
\begin{tikzpicture}[scale=1]
\def \x {4.2}
\def \y {2.75}
\def \xi {0}
\def \yi {0}
\def \a {7}

\node (label) at (\xi, \yi) {\includegraphics[height=6cm]{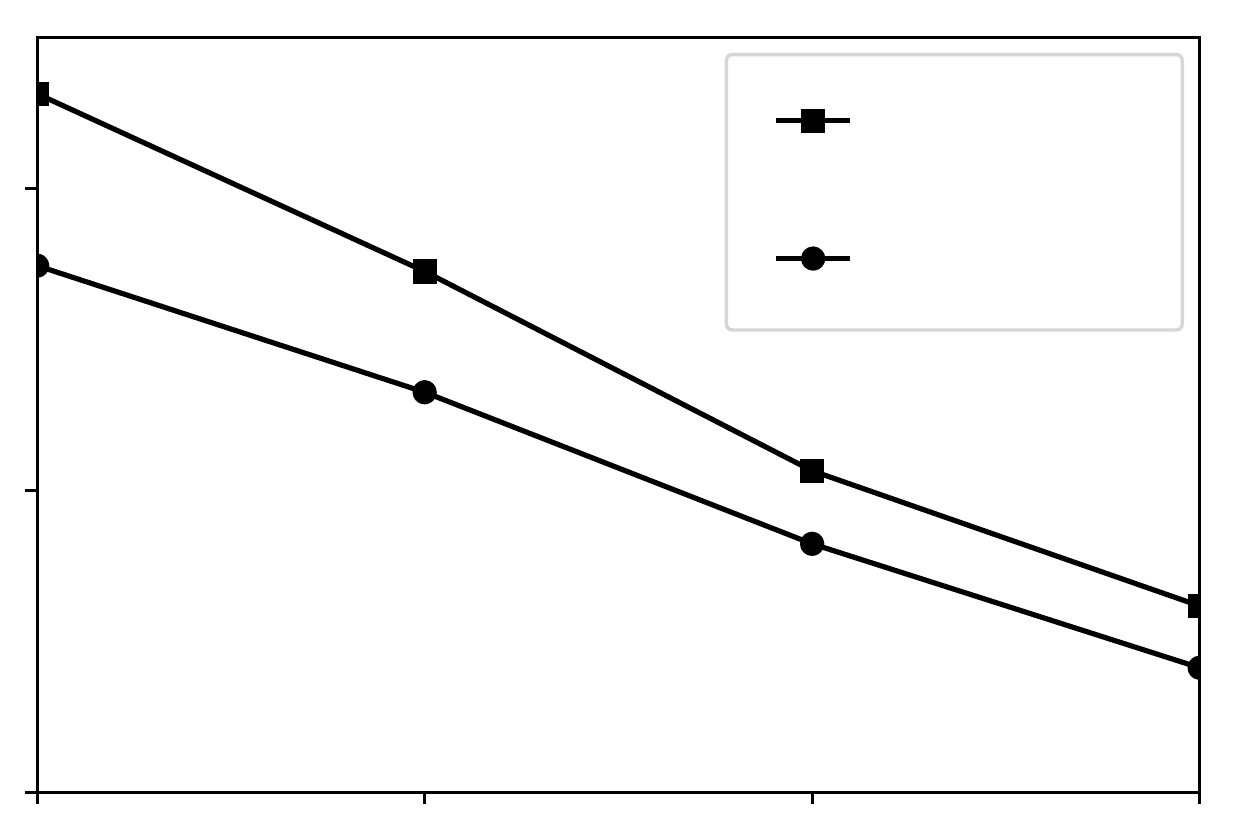}};
\draw (-\x, -\y) node[below] {$0$}; 
\draw (-\x/3, -\y) node[below] {$1$}; 
\draw (\x/3, -\y) node[below] {$2$};
\draw (\x, -\y) node[below] {$3$} node[right] {number of folds}; 
\draw (-\x, -\y) node[left] {$0$};
\draw (-\x, -.2*\y) node[left] {$1$};
\draw (-\x, .6*\y) node[left] {$2$};
\draw (-\x, \y) node[anchor = south west] {$e_\e \times 10^3$};
\draw (2, 2.13) node[right] {$\phi = \tfrac{2\pi}5$};
\draw (2, 1.13) node[right] {$\phi = \tfrac{2\pi}7$};

\end{tikzpicture}
\caption{The energy values $e_\e = E_\e (\overline U_\e)$ for $\e = 2^{-2}$ computed from folded initial conditions (see Figure \ref{fig:fold}). 
}
\label{fig:Efolds}
\end{figure}

%

\begin{figure}[h]
\centering
\begin{tikzpicture}[scale=1]
\def \x {2.3}
\def \y {-2.8}
\def \xi {0}
\def \yi {0}
\def \a {7}
\begin{scope}[shift={(0,0)},scale=1]
  \node (label) at (\xi, \yi) {\includegraphics[height=6cm]{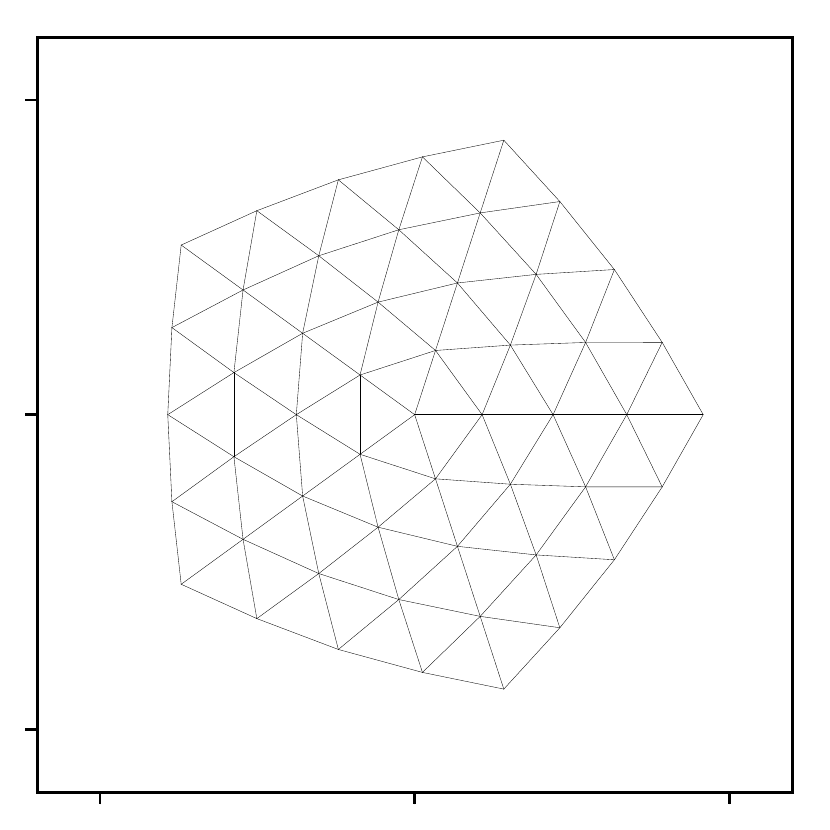}};
  \draw (-\x, \y) node[below] {$-1$};  
  \draw (0, \y) node[below] {$0$};
  \draw (\x, \y) node[below] {$1$};
  \draw (\y, -\x) node[left] {$-1$};  
  \draw (\y, 0) node[left] {$0$};
  \draw (\y, \x) node[left] {$1$};
  \draw (-\x,\x) node[right] {$0$ folds};
\end{scope}
\begin{scope}[shift={(\a,0)},scale=1]
  \node (label) at (\xi, \yi){\includegraphics[height=6cm]{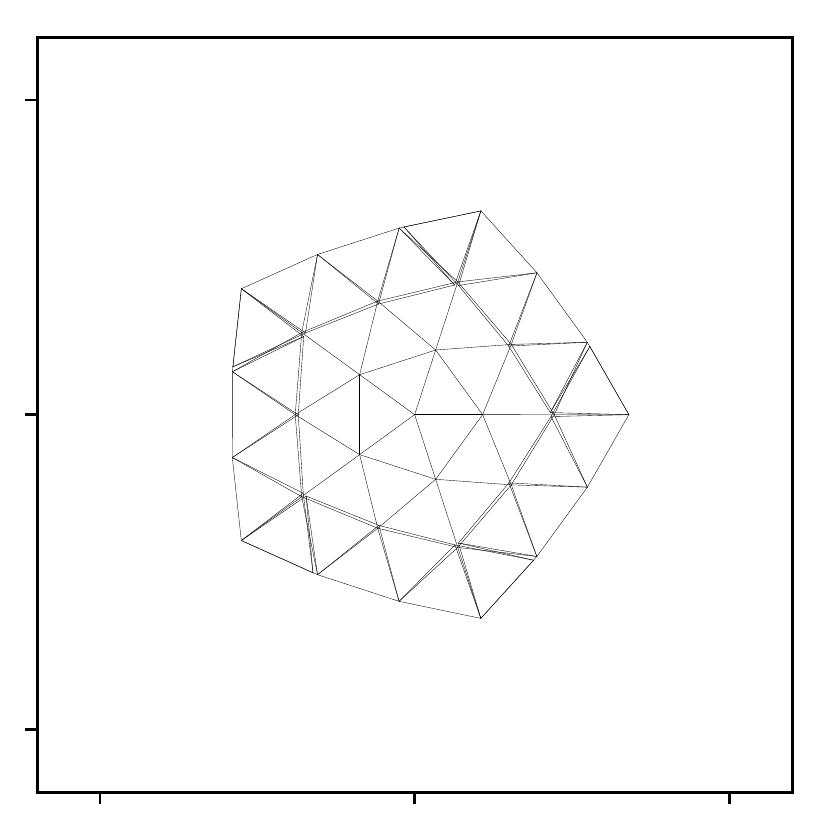}};
  \draw (-\x, \y) node[below] {$-1$};  
  \draw (0, \y) node[below] {$0$};
  \draw (\x, \y) node[below] {$1$};
  \draw (\y, -\x) node[left] {$-1$};  
  \draw (\y, 0) node[left] {$0$};
  \draw (\y, \x) node[left] {$1$};
  \draw (-\x,\x) node[right] {$1$ fold};
\end{scope}
\begin{scope}[shift={(0,-\a)},scale=1]
  \node (label) at (\xi, \yi) {\includegraphics[height=6cm]{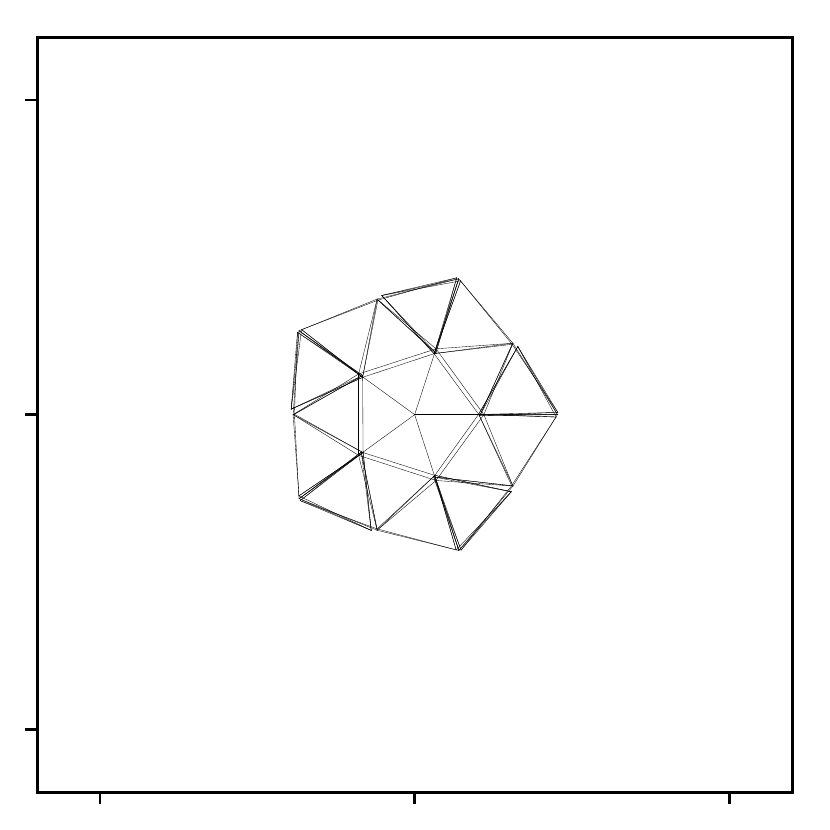}};
  \draw (-\x, \y) node[below] {$-1$};  
  \draw (0, \y) node[below] {$0$};
  \draw (\x, \y) node[below] {$1$};
  \draw (\y, -\x) node[left] {$-1$};  
  \draw (\y, 0) node[left] {$0$};
  \draw (\y, \x) node[left] {$1$};
  \draw (-\x,\x) node[right] {$2$ folds};
\end{scope}
\begin{scope}[shift={(\a,-\a)},scale=1]
  \node (label) at (\xi, \yi){\includegraphics[height=6cm]{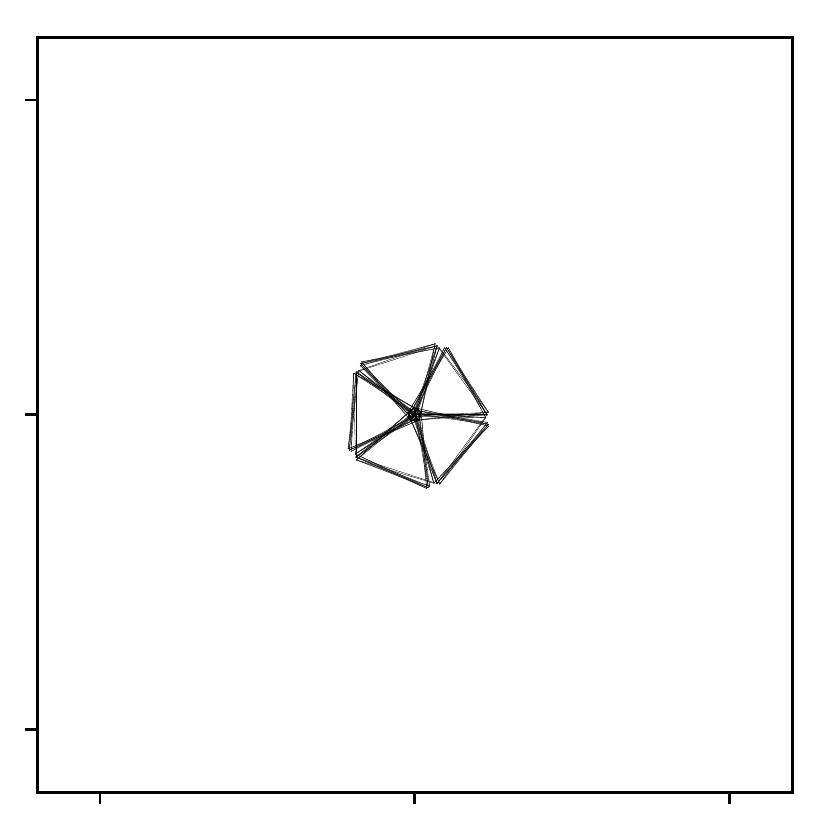}};
  \draw (-\x, \y) node[below] {$-1$};  
  \draw (0, \y) node[below] {$0$};
  \draw (\x, \y) node[below] {$1$};
  \draw (\y, -\x) node[left] {$-1$};  
  \draw (\y, 0) node[left] {$0$};
  \draw (\y, \x) node[left] {$1$};
  \draw (-\x,\x) node[right] {$3$ folds};
\end{scope}
\end{tikzpicture}
\caption{Minimizers of $E_\varepsilon$ obtained from the folded initial conditions shown in Figure \ref{fig:fold}.}
\label{fig:5fold}
\end{figure}

\clearpage

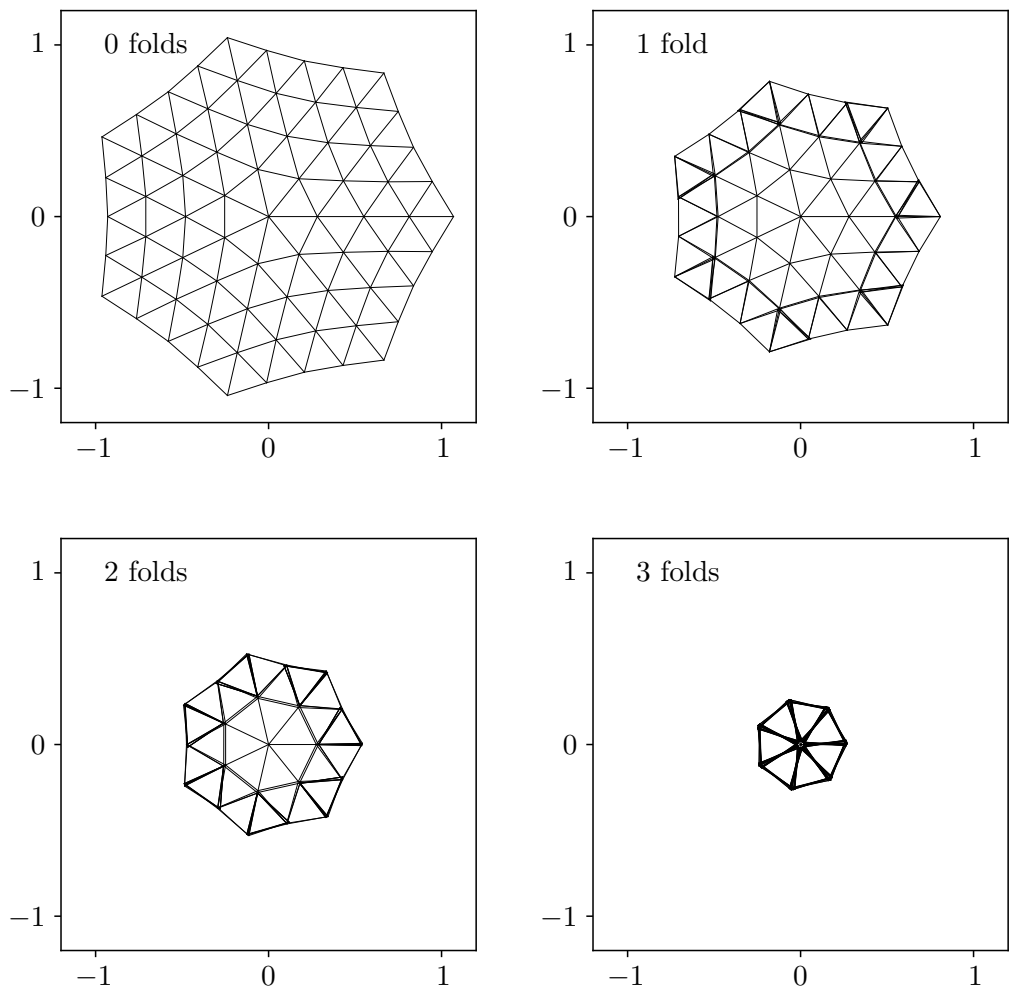
\begin{figure}[h]
\centering
\begin{tikzpicture}[scale=1]
\def \x {2.3}
\def \y {-2.8}
\def \xi {0}
\def \yi {0}
\def \a {7}
\begin{scope}[shift={(0,0)},scale=1]
  \node (label) at (\xi, \yi) {\includegraphics[height=6cm]{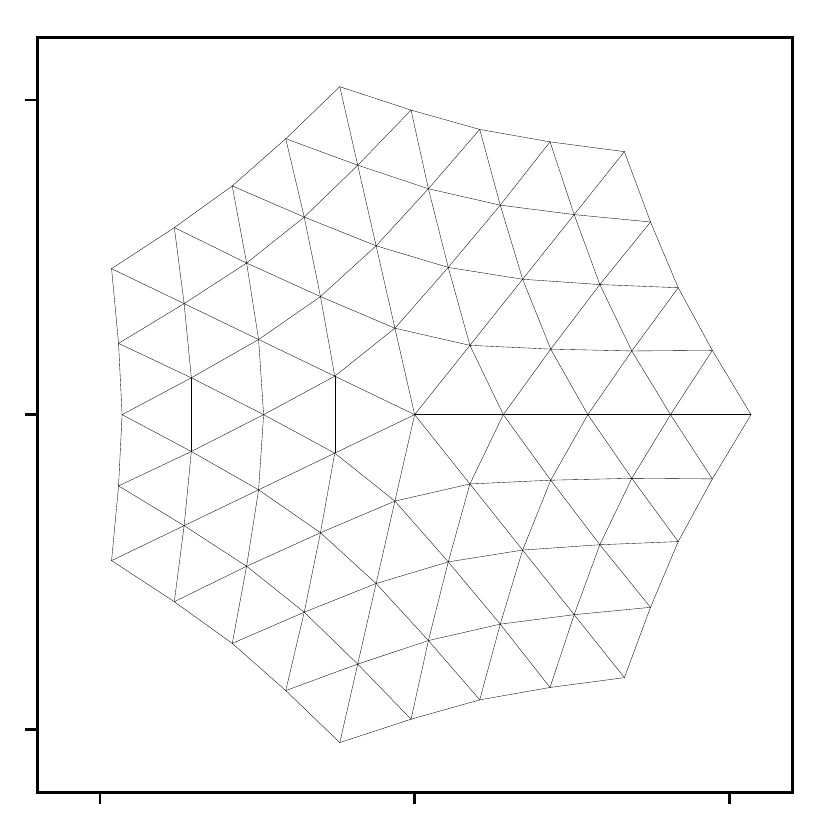}};
  \draw (-\x, \y) node[below] {$-1$};  
  \draw (0, \y) node[below] {$0$};
  \draw (\x, \y) node[below] {$1$};
  \draw (\y, -\x) node[left] {$-1$};  
  \draw (\y, 0) node[left] {$0$};
  \draw (\y, \x) node[left] {$1$};
  \draw (-\x,\x) node[right] {$0$ folds};
\end{scope}
\begin{scope}[shift={(\a,0)},scale=1]
  \node (label) at (\xi, \yi){\includegraphics[height=6cm]{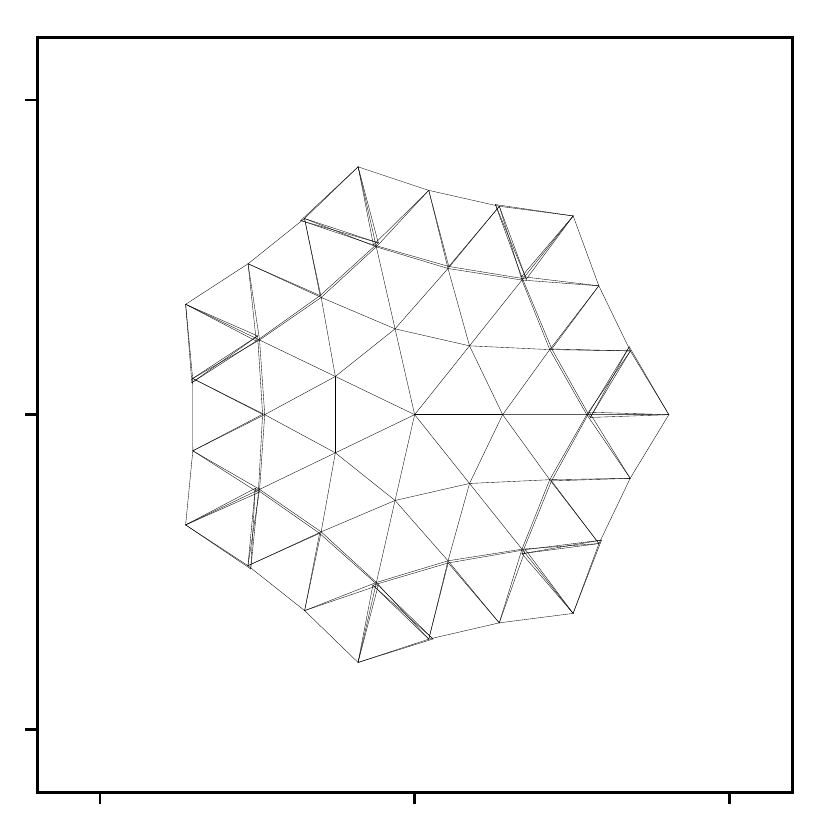}};
  \draw (-\x, \y) node[below] {$-1$};  
  \draw (0, \y) node[below] {$0$};
  \draw (\x, \y) node[below] {$1$};
  \draw (\y, -\x) node[left] {$-1$};  
  \draw (\y, 0) node[left] {$0$};
  \draw (\y, \x) node[left] {$1$};
  \draw (-\x,\x) node[right] {$1$ fold};
\end{scope}
\begin{scope}[shift={(0,-\a)},scale=1]
  \node (label) at (\xi, \yi) {\includegraphics[height=6cm]{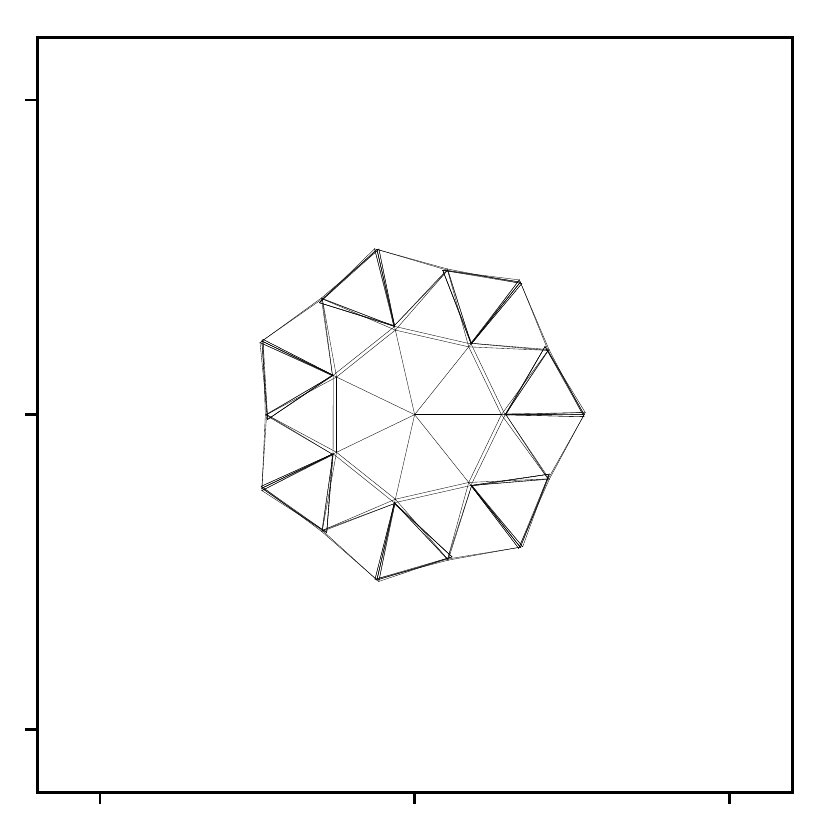}};
  \draw (-\x, \y) node[below] {$-1$};  
  \draw (0, \y) node[below] {$0$};
  \draw (\x, \y) node[below] {$1$};
  \draw (\y, -\x) node[left] {$-1$};  
  \draw (\y, 0) node[left] {$0$};
  \draw (\y, \x) node[left] {$1$};
  \draw (-\x,\x) node[right] {$2$ folds};
\end{scope}
\begin{scope}[shift={(\a,-\a)},scale=1]
  \node (label) at (\xi, \yi){\includegraphics[height=6cm]{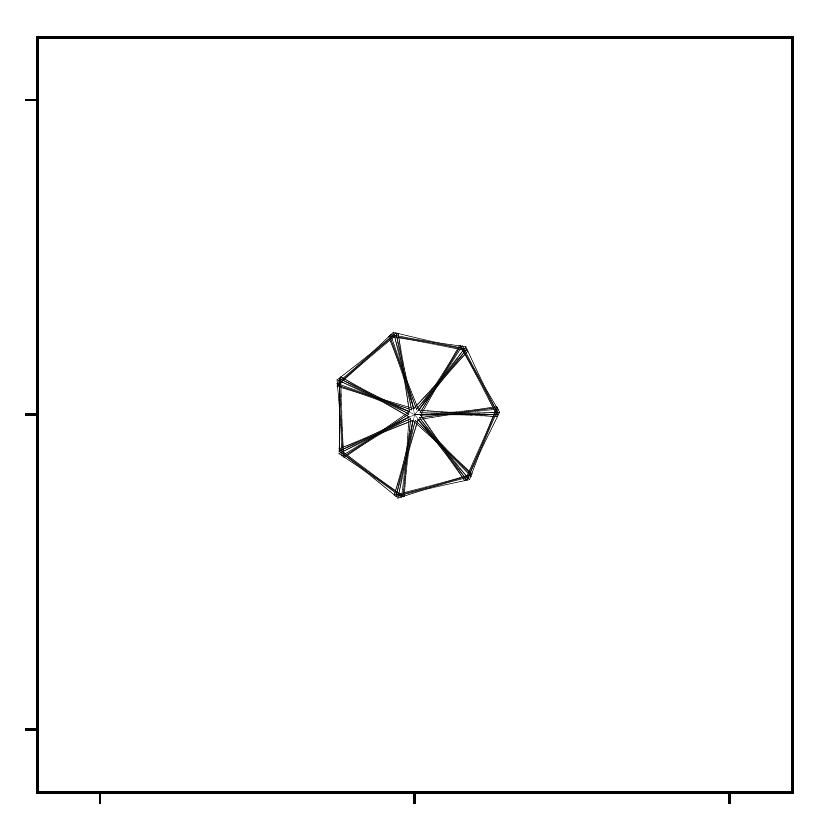}};
  \draw (-\x, \y) node[below] {$-1$};  
  \draw (0, \y) node[below] {$0$};
  \draw (\x, \y) node[below] {$1$};
  \draw (\y, -\x) node[left] {$-1$};  
  \draw (\y, 0) node[left] {$0$};
  \draw (\y, \x) node[left] {$1$};
  \draw (-\x,\x) node[right] {$3$ folds};
\end{scope}
\end{tikzpicture}
\caption{Minimizers of $E_\varepsilon$ obtained from the folded initial conditions shown in Figure \ref{fig:fold}.}
\label{fig:7fold}
\end{figure}

\clearpage

\section{Future perspectives}

We develop a continuum model (the energy $E$ in \eqref{E:QW}) for the energy of a disclination in a planar geometry. We construct $E$ rigorously from an atomistic model (Theorem \ref{thm}), and show that any minimizer of $E$ corresponds to a stressed configuration of the medium (Proposition \ref{p:Egtr0}). We intend this first rigorous study as a stepping stone towards a general variational theory capable of describing both morphology and energetics of disclinations and their effect on the macrosocopic properties of a material.
By analyzing the interaction of  mismatches and distortions on the hexagonal lattice,
we take a first step towards the investigation of the interaction of defects with  the lattice kinematics, thus opening a possible path to modeling more complex systems such as austenite-martensite microstructures in Shape-Memory Alloys and kink formations as mentioned in the introduction.

  \section*{Acknowledgements}
P.C. is supported by JSPS  Grant-in-Aid for Young Scientists (B) 16K21213
and   by JSPS Innovative Area Grant 19H05131. P.C. holds an honorary appointment at La Trobe University and is a member of GNAMPA. 
The authors 
acknowledge the Research Institute for Mathematical Sciences, an International Joint Usage and Research Center located in Kyoto University, where part of the work contained in this paper was carried out.

 \appendix
 \section{Appendix}

\begin{lem}\label{2005071702}
For all $0 \leq \sigma_1 \leq \sigma_2$ and all $\theta \in \R$,
\begin{eqnarray}\label{2005171621}
14 \sum_{k=0}^5 (|\Sigma R_{\theta + k\pi/3} e_1| - 1)^2
   \geq (\sigma_1 - 1)^2 + (\sigma_2 - 1)^2,
   \quad \text{where } \Sigma := \begin{bmatrix}
     \sigma_1 & 0 \\ 0 & \sigma_2
   \end{bmatrix}.
\end{eqnarray}
\end{lem}

\begin{proof}
We separate 2 cases: $\sigma_1^2 + \sigma_2^2 \geq 2$ and $\sigma_1^2 + \sigma_2^2 \leq 2$. In the first case, we start with bounding the right-hand side 
of (\ref{2005171621})  
from above. If $\sigma_1 \geq 1$, then clearly $(\sigma_1 - 1)^2 + (\sigma_2 - 1)^2 \leq 2 (\sigma_2 - 1)^2$. If $\sigma_1 \leq 1$, then from the observation that the line segment between the points $(0, \sqrt 2)$ and $(1,1)$ on the circle $\sqrt 2 \mathbb S$   is below the arc of $\sqrt 2 \mathbb S$ between the same two points, we obtain that $\sigma_2 \geq 1 + (\sqrt 2 - 1) (1 - \sigma_1)$. Hence,
\[
 0 \leq  1 - \sigma_1 \leq \frac{\sigma_2 - 1}{\sqrt 2 - 1} = (\sqrt 2 + 1) (\sigma_2 - 1).
\]
This together with the bound for $\sigma_1 \geq 1$ we get
\begin{equation} \label{pf:2}
  (\sigma_1 - 1)^2 + (\sigma_2 - 1)^2 < 7 (\sigma_2 - 1)^2.
\end{equation}

We continue by bounding the left-hand side  in (\ref{2005171621}) from below. Writing
\[
  |\Sigma R_{\theta + k\pi/3} e_1|^2 
  = \sigma_1^2 \cos^2 (\theta + k\pi/3) + \sigma_2^2 \sin^2 (\theta + k\pi/3),
\]
we note from the facts that $\alpha \mapsto \sin^2 \alpha$ is $\pi$-periodic and $\{ \alpha \in [0, \pi]: \sin^2 \alpha \geq \frac34 \} = [\pi/3, 2\pi/3]$ that there are at least two values for $k \in \{ 0, \ldots, 5 \}$ for which $\sin^2 (\theta + k\pi/3) \geq \frac34$. For these values of $k$, we have by $\sigma_2 \geq \max \{1, \sigma_1 \}$ and $\sigma_1^2 + \sigma_2^2 \geq 2$ that
\[
  |\Sigma R_{\theta + k\pi/3} e_1|^2 
  \geq \frac14 \sigma_1^2 + \frac34 \sigma_2^2
  \geq \frac12 \sigma_2^2 + \frac12
  \geq 1,
\]
and thus
\begin{equation} \label{pf:1}
  14 \sum_{k=0}^5 (|\Sigma R_{\theta + k\pi/3} e_1| - 1)^2
  \geq 28 \Big( \frac1{\sqrt 2} \sqrt{ \sigma_2^2 + 1 } - 1 \Big)^2.  
\end{equation}
Estimating the convex function $f(x) = \sqrt{ x^2 + 1 }$ from below by its tangent at $x = 1$ yields
\begin{align*}
  \sqrt{ \sigma_2^2 + 1 }
  \geq \sqrt{ 2 } + (\sigma_2 - 1) / \sqrt{ 2 }.
\end{align*}
Plugging this estimate into \eqref{pf:1}, we observe that the resulting lower bound equals the upper bound in \eqref{pf:2}. This completes the proof in the case $\sigma_1^2 + \sigma_2^2 \geq 2$.

In the second case, $\sigma_1^2 + \sigma_2^2 \leq 2$, we follow a similar procedure. 
We claim that $(\sigma_2 - 1)^2 \leq (\sigma_1 - 1)^2$.
Then, instead of \eqref{pf:2}, we get
\begin{equation} \label{pf:4}
  (\sigma_1 - 1)^2 + (\sigma_2 - 1)^2 \leq 2 (1 - \sigma_1)^2.
  \end{equation}

To prove this claim, we note from $\sigma_1 \leq \sigma_2$ that $\sigma_1 - 1 \leq \sigma_2 - 1$. To get an upper bound for $\sigma_2 - 1$, we obtain from  $\sigma_1^2 + \sigma_2^2 \leq 2$ that 
\[ \sigma_2 \leq \sqrt{2 - \sigma_1^2}. \]
Since the right-hand side is concave in $\sigma_1$, we can bound it from above by its tangent at $\sigma_1=1$, this yields $\sigma_2 \leq 2 - \sigma_1$, and thus $\sigma_2 - 1 \leq 1 - \sigma_1$. The claim follows.

%
    %

For the left-hand side
of (\ref{2005171621}),
similar to the previous case, there are at least two values for $k \in \{ 0, \ldots, 5 \}$ for which $\cos^2 (\theta + k\pi/3) \geq \frac34$. For these values of $k$, we have by $\sigma_1 \leq \sigma_2$, $\sigma_1^2 + \sigma_2^2 \leq 2$ and $\sigma_1 \leq 1$ that 
\[
  |\Sigma R_{\theta + k\pi/3} e_1|^2 
  \leq \frac34 \sigma_1^2 + \frac14 \sigma_2^2
  \leq \frac12 \sigma_1^2 + \frac12
  \leq 1,
\]
and thus
\begin{equation}\label{pf:3}
  14 \sum_{k=0}^5 (1 - |\Sigma R_{\theta + k\pi/3} e_1|)^2 
  \geq 28 \Big( 1 - \frac1{\sqrt 2} \sqrt{ \sigma_1^2 + 1 } \Big)^2.  
\end{equation}
%
%
Using that $f(x) = \sqrt{ x^2 + 1 } \leq f(1) + (1-x) (f(0) - f(1))$, we obtain
\begin{equation*}
  \sqrt{ \sigma_1^2 + 1 } 
  \leq \sqrt 2 - (1 - \sigma_1) (\sqrt 2 - 1).
\end{equation*}
Plugging this estimate into \eqref{pf:3}, we observe that the resulting lower bound is larger than the upper bound in \eqref{pf:4}. 
\end{proof}
 
 \begin{lem}\label{l:QW0}
 Let $W$ be as in (\ref{W:KM17}) with $\Psi=0$, that is, 
  \begin{equation*} 
  W(A) = \sum_{e \in B^1} \Phi \left( \big| e A \big| - 1 \right).
\end{equation*}   
Then, 
$$
QW(0)=0.
$$
 \end{lem}
 
 \begin{proof}
 Since $0\le W(A)$ for all $A\in \mathbb{R}^{2 \times 2}$, it follows $0\le QW(0)$.
 We are left to show the reverse inequality.
To do this, we introduce $RW:\mathbb{R}^2\to\mathbb{R}$, the rank-1 convex envelope of $W$ (see \cite[Sec. 6.4]{Dac08})
which is characterized by the following formula (\cite[Thm. 6.10]{Dac08})
 \begin{multline}\label{2005131248}
RW(A)=\inf \bigg\{ \sum_{i=1}^I\lambda_iW(A_i): \\
I\in\mathbb{N}, \ \lambda_i \geq 0, \ \sum^{I}_{i=1} \lambda_i=1, \
\sum_{i=1}^I\lambda_iA_i=A, \ (\lambda_i,A_i) \textrm{ satisfy } (H_I) \bigg\},
 \end{multline}
where $(H_I)$ is the  
hierarchical compatibility
constraint defined in
 $\cite[\textrm{Definition 5.14}]{Dac08}$.  
Recall also that for $W:\mathbb{R}^2\to\mathbb{R}$, then $QW\le RW$ (see \cite[Sec. 6.1]{Dac08}). Hence, it remains to show that $RW(0) \leq 0$.

To show that $RW(0) \leq 0$, we take in \eqref{2005131248} $I = 4$, $\lambda_i = \frac14$ and 
consider the family 
\begin{eqnarray*}
A_1 := \begin{bmatrix}
      1 & 0 \\
       0 & -1
   \end{bmatrix},
   A_2 := \begin{bmatrix}
      1 & 0 \\
       0 & 1
   \end{bmatrix} ,
   A_3=-A_1,A_4=-A_2
\end{eqnarray*}
and observe that
\begin{enumerate}
\item[1)] $ \frac{1}{4}A_1+\frac{1}{4}A_2+\frac{1}{4}A_3+\frac{1}{4}A_4=0$;
\item[2)] $\textrm{rank}(A_1-A_2)=\textrm{rank}(A_3-A_4)=1$;
\item[3)] $\textrm{rank}(\frac{A_1+A_2}{2}-\frac{A_3+A_4}{2})=1$;
\item[4)] $|A_ie_j|^2 = e_j^TA_i^TA_ie_j=1$, $j=1,\dots,6$ 
where 
$e_j\in B^1$ and $i=1,\dots,4$.
\end{enumerate}
 By \cite[Example 5.15]{Dac08}, Properties $2)-3)$  imply that
$(\lambda_i,A_i)$
satisfy condition $(H_I)$.
Then, it is then easy to see from 1) that  $(\lambda_i,A_i)$ constitute an admissible candidate in the minimization problem
(\ref{2005131248}).
Finally, condition $4)$ implies $W(A_i)=0$ for $i=1,\dots,4.$ 

Collecting all the results above, we   have
$$
0\leq QW(0)\le RW(0)\le\sum_{i=1}^4\lambda_i W(A_i)=0.
$$
 \end{proof}

\bibliographystyle{alpha} 
\bibliography{refsPatrick}

\end{document}